\pdfoutput=1
 \documentclass[11pt]{article}
 \def\arXiv{1}
\usepackage{tablefootnote}

\newcommand{\notarxiv}[1]{foo}
\newcommand{\arxiv}[1]{ba}
\ifdefined \arXiv
\renewcommand{\arxiv}[1]{#1}%
\renewcommand{\notarxiv}[1]{\ignorespaces}%
\else%
\renewcommand{\arxiv}[1]{\ignorespaces}%
\renewcommand{\notarxiv}[1]{#1}%
\fi%

\usepackage{thmtools}
\usepackage{thm-restate}
\usepackage{bbm}
\usepackage{mathtools}

 \notarxiv{
 	\undef\corollary
	 \undef\definition
	 \undef\assumption

	\setcitestyle{numbers,square,comma} 
	\declaretheorem[name=Lemma,sibling=theorem]{lem}
	\declaretheorem[name=Proposition,sibling=theorem]{prop}
	\declaretheorem[name=Corollary,sibling=theorem]{corollary}
	\declaretheorem[name=Assumption,sibling=theorem]{assumption}
	\declaretheorem[name=Definition,sibling=theorem]{definition}
    \declaretheorem[name=Fact,sibling=theorem]{fact}

	\undef\lemma
	\undef\proposition	
}

\arxiv{
	\usepackage{amsthm}
	\usepackage[linesnumbered,ruled,vlined,boxed,algo2e]{algorithm2e}

	\usepackage[dvipsnames,table,xcdraw]{xcolor}
	\usepackage[top=1in, right=1in, left=1in, bottom=1in]{geometry}
	\usepackage[numbers,square,sort&compress]{natbib}
	\usepackage{url}
	\usepackage[hidelinks]{hyperref}
	\hypersetup{
		colorlinks=true,
		linkcolor=blue!70!black,
		citecolor=blue!70!black,
		urlcolor=blue!70!black}
	
	\theoremstyle{plain}
	
	\newtheorem{theorem}{Theorem}
	\newtheorem{lem}{Lemma}
	
	\newtheorem{assumption}{Assumption}
	
	\newtheorem{corollary}{Corollary}
	\newtheorem{fact}{Fact}
	\newtheorem{definition}{Definition}
	
	\theoremstyle{definition}
	\newtheorem{remark}{Remark}
	
	\newtheorem*{example*}{Example}

}

\usepackage{enumitem}
\usepackage{amssymb}
\usepackage{amsbsy}
\usepackage{amsmath}

\usepackage{cleveref}
\Crefname{assumption}{Assumption}{Assumptions}
\Crefname{fact}{Fact}{Facts}
\Crefname{lem}{Lemma}{Lemmas}
\Crefname{prop}{Proposition}{Propositions}

\newtheorem{innercustomgeneric}{\customgenericname}
\providecommand{\customgenericname}{}
\newcommand{\newcustomtheorem}[2]{%
  \newenvironment{#1}[1]
  {%
   \ifdefined\crefalias\crefalias{innercustomgeneric}{#2}\fi
   \renewcommand\customgenericname{#2}%
   \renewcommand\theinnercustomgeneric{##1}%
   \innercustomgeneric
  }
  {\endinnercustomgeneric}%
  \ifdefined\crefname\crefname{#2}{#2}{#2s}\fi
}

\newcustomtheorem{customprop}{Proposition}

\usepackage{amsfonts}
\usepackage{latexsym}
\usepackage{graphicx}
\usepackage{color}
\usepackage{xifthen}
\usepackage{xspace}
\usepackage{mathtools}
\usepackage{bbm}
\usepackage{multirow}
\usepackage[normalem]{ulem}
\usepackage{array}
\usepackage[utf8]{inputenc}
\usepackage[T1]{fontenc}
\usepackage{etoolbox}
\usepackage{environ}
\newtoggle{restatements}

\newtoggle{heavyplots}

\usepackage{xfrac}
\usepackage{xargs}

\usepackage{titlesec}

\usepackage{setspace}

\usepackage{esvect}

\usepackage{nicefrac}

\usepackage{cases}
\usepackage{empheq}

 \usepackage{booktabs}
 \usepackage{multirow}
 
 \usepackage{caption}

\usepackage{amsmath}

\DeclarePairedDelimiter{\abs}{\lvert}{\rvert} %
\DeclarePairedDelimiter{\brk}{[}{]}
\DeclarePairedDelimiter{\crl}{\{}{\}}
\DeclarePairedDelimiter{\prn}{(}{)}

\DeclarePairedDelimiter{\norm}{\|}{\|}

\DeclarePairedDelimiter{\ceil}{\lceil}{\rceil}

\DeclarePairedDelimiterXPP{\inner}[2]{}{\langle}{\rangle}{}{#1,#2}

\newcommand{\overeq}[1]{\overset{#1}{=}}
\newcommand{\overle}[1]{\overset{#1}{\le}}

\NewDocumentCommand\Ex{s O{} m }{%
	\mathbb{E}%
	\begingroup
	\IfBooleanTF{#1}
	{\ExInn*{#3}}
	{\ExInn[#2]{#3}}%
	\endgroup
}

\DeclarePairedDelimiterX\ExInn[1]{[}{]}{%
	\activatebar
	#1%
}

\RenewDocumentCommand\Pr{sO{}r()}{%
	\mathbb{P}%
	\begingroup
	\IfBooleanTF{#1}
	{\PrInn*{#3}}
	{\PrInn[#2]{#3}}%
	\endgroup
}

\DeclarePairedDelimiterX\PrInn[1](){%
	\activatebar
	#1%
}

\newcommand{\activatebar}{%
	\begingroup\lccode`~=`|
	\lowercase{\endgroup\def~}{\;\delimsize\vert\;}%
	\mathcode`|=\string"8000
}

\newcommand\numberthis{\addtocounter{equation}{1}\tag{\theequation}}

\newcommand{\R}{\mathbb{R}} %
\newcommand{\N}{\mathbb{N}} %
\newcommand{\E}{\mathbb{E}} %
\renewcommand{\P}{\mathbb{P}}	%

\SetCommentSty{mycommfont}
\SetKwInput{KwInput}{Input} 
\SetKwInput{KwParameter}{Parameters}                %
\SetKwInput{KwOutput}{Output}              %
\SetKwInput{KwReturn}{Return}              %
\SetKwComment{Comment}{$\triangleright$\ }{}
\SetKwInOut{KwParameters}{Parameters}
\SetCommentSty{mycommfont}
\makeatletter
\long\def\@makecaption#1#2{
  \vskip 0.8ex
  \setbox\@tempboxa\hbox{\small {#1:} #2}
  \parindent 1.5em  %
  \dimen0=\hsize
  \advance\dimen0 by -3em
  \ifdim \wd\@tempboxa >\dimen0
  \hbox to \hsize{
    \parindent 0em
    \hfil 
    \parbox{\dimen0}{\def\baselinestretch{0.96}\small
      {#1.} #2
    } 
    \hfil}
  \else \hbox to \hsize{\hfil \box\@tempboxa \hfil}
  \fi
}
\makeatother

\DeclareMathOperator*{\argmin}{arg\,min}

\providecommand{\abs}{\mathop{\rm abs}}

\newcommand{\indic}[1]{\mathbbm{1}_{\crl{#1}}}
\newcommand{\indicb}[1]{\mathbbm{1}_{\crl*{#1}}}
\newcommand{\del}{\partial}
\newcommand{\grad}{\nabla}
\newcommand{\simiid}{\stackrel{\textrm{\tiny iid}}{\sim}}

\newcommand{\half}{\frac{1}{2}}

\newcommand{\defeq}{\coloneqq}

\newcommand{\xset}{\mathcal{X}}
\newcommand{\sset}{\mathcal{S}}

\newcommand{\uniform}{\mathsf{Unif}}  %
\newcommand{\bernoulli}{\mathsf{Bernoulli}}  %

\newcommand{\eps}{\varepsilon}
\newcommand{\veps}{\varepsilon}

\newcommand{\xopt}{x^\star}

\renewcommand{\O}[1]{O\left( #1 \right)}
\newcommand{\Omeg}[1]{\Omega\left( #1 \right)}

\newcommand{\Otil}[1]{\widetilde{O}( #1 )}

\newcommand{\Otilb}[1]{\widetilde{O}\left( #1 \right)}

\newcommand{\alg}{\mathsf{alg}}
\newcommand{\Aclass}[1][]{\mathcal{A}_{\mathsf{#1}}}
\newcommand{\Iclass}[1][]{\mathcal{I}_{\mathsf{#1}}}
\newcommand{\Mclass}[1][]{\mathcal{M}_{\mathsf{#1}}}

\newcommand{\Err}[3][T]{\mathsf{Err}_{#1}\prn*{#2,#3}}
\newcommand{\ErrBlank}[1][T]{\mathsf{Err}_{#1}}

\newcommand{\ErrHP}[3][T]{\mathsf{Err}_{#1}^{\delta}\prn*{#2,#3}}
\newcommand{\ErrE}[3][T]{\mathsf{Err}_{#1}^{\E}\prn*{#2,#3}}
\newcommand{\ErrCustom}[4][T]{\mathsf{Err}_{#1}^{\textup{\tiny#2}}\prn*{#3,#4}}

\newcommand{\Aall}{\Aclass[SO]}
\newcommand{\AFO}{\Aclass[SFO]}

\newcommand{\PoAGen}[4][T]{\mathsf{PoA}_{#1}\prn*{#2,#3;#4}}
\newcommand{\PoA}[3][T]{\mathsf{PoA}_{#1}\prn*{#2,#3}}
\newcommand{\bigPoA}[3][T]{\mathsf{PoA}_{#1}\prn[\big]{#2,#3}}
\newcommand{\PoAE}[3][T]{\mathsf{PoA}_{#1}^{\E}\prn*{#2,#3}}
\newcommand{\bigPoAE}[3][T]{\mathsf{PoA}_{#1}^{\E}\prn[\big]{#2,#3}}

\newcommand{\PoAHP}[3][T]{\mathsf{PoA}_{#1}^{\delta}\prn*{#2,#3}}

\newcommand{\Lip}{Lip}
\newcommand{\MSLip}{SM\text{-}Lip}
\newcommand{\SMLip}{SM\text{-}Lip}

\newcommand{\Lmin}{l_1}
\newcommand{\Lmax}{l_2}
\newcommand{\Rmin}{r_1}
\newcommand{\Rmax}{r_2}

\newcommand{\IclassPlus}{\mathcal{I}_{1}}
\newcommand{\IclassMinus}{\mathcal{I}_{0}}
\newcommand{\IclassV}{\mathcal{I}_{v}}
\newcommand{\IclassK}{\mathcal{I}_{k}}

\newcommand{\Pplus}{P_{1}}
\newcommand{\Pminus}{P_{0}}
\newcommand{\Pv}{P_{v}}

\newcommand{\fk}{f_k}
\newcommand{\fK}{f_K}

\newcommand{\Pone}{P_1}
\newcommand{\Pzero}{P_0}

\newcommand{\yairside}[1]{\todo[color=blue!10]{Yair: #1}}
\newcommand{\yair}[1]{{\bf \color{blue} Yair: #1}}
\newcommand{\ollie}[1]{{\bf \color{red} Ollie: #1}}

\renewcommand{\yairside}[1]{\ignorespaces}
\renewcommand{\yair}[1]{\ignorespaces}
\renewcommand{\ollie}[1]{\ignorespaces} 
\notarxiv{
	\titlespacing*{\section}
	{0pt}{8pt plus 2pt minus 1pt}{5pt plus 0pt minus 3pt}
	\titlespacing*{\subsection}
	{0pt}{6pt plus 0pt minus 2pt}{5pt plus 0pt minus 3pt}
	\titlespacing*{\paragraph}
	{0pt}{4pt plus 4pt minus 0pt}{6pt}
}
\arxiv{
	\titlespacing*{\paragraph}
	{0pt}{8pt plus 2pt minus 2pt}{6pt}
}

\title{The Price of Adaptivity in Stochastic Convex Optimization}

\notarxiv{
\usepackage{times}
}

\notarxiv{
\coltauthor{%
 \Name{Yair Carmon} \Email{ycarmon@tauex.tau.ac.il}\\
 \Name{Oliver Hinder} \Email{ohinder@pitt.edu}\\
}
}

\arxiv{
	\title{The Price of Adaptivity in Stochastic Convex Optimization\footnote{Accepted for presentation at the
	Conference on Learning Theory (COLT) 2024.}}
	\author{Yair Carmon\\
		\href{mailto:ycarmon@tauex.tau.ac.il}{\texttt{ycarmon@tauex.tau.ac.il}}\and Oliver Hinder \\
		\href{ohinder@pitt.edu}{\texttt{ohinder@pitt.edu}}
}
\date{}
}

\begin{document}

\maketitle

\begin{abstract}
We prove impossibility results for adaptivity in non-smooth stochastic convex optimization. Given a set of problem parameters we wish to adapt to, we define a ``price of adaptivity'' (PoA) that, roughly speaking, measures the multiplicative increase in suboptimality due to uncertainty in these parameters. When the initial distance to the optimum is unknown but a gradient norm bound is known, we show that the PoA is at least logarithmic for expected suboptimality, and double-logarithmic for median suboptimality. When there is uncertainty in both distance and gradient norm, we show that the PoA must be polynomial in the level of uncertainty. Our lower bounds nearly match existing upper bounds, and establish that there is no parameter-free lunch.

En route, we also establish tight upper and lower bounds for (known-parameter) high-probability stochastic convex optimization with heavy-tailed and bounded noise, respectively. 
\end{abstract}

\notarxiv{
}

\section{Introduction}

Stochastic optimization methods in modern machine learning often require carefully tuning sensitive algorithmic parameters at significant cost in time, computation, and expertise. This reality has led to sustained interest in developing adaptive (or parameter-free) algorithms that require minimal or no tuning \cite{kavis2019unixgrad,orabona2014simultaneous,vaswani2019painless,mcmahan2017survey,zhang2022pde,zhang2022parameter,chen2022better,orabona2021parameter,carmon2022making,orabona2016coin,orabona2021modern,cutkosky2018black,paquette2020stochastic,mhammedi2019lipschitz,sharrock2023coin,jacobsen2023unconstrained,ivgi2023dog}. However, a basic theoretical question remains open: Are existing methods ``as adaptive as possible,'' or is there substantial room for improvement? Put differently, is there a fundamental price to be paid (in terms of rate of convergence) for not knowing the problem parameters in advance?

To address these questions, we must formally define what it means for an adaptive algorithm to be efficient. 
The standard notion of minimax optimality \cite{agarwal2012information} does not suffice, since it does not constrain the algorithm to be agnostic to the parameters defining the function class; stochastic gradient descent (SGD) is in many cases minimax optimal, but its step size requires problem-specific tuning. 
To motivate our solution, we observe that guarantees for adaptive algorithms admit the following interpretation: assuming that the input problem satisfies certain assumptions (e.g., Lipschitz continuity, smoothness, etc.) the adaptive algorithm attains performance close to the best performance that is possible to guarantee given only these assumptions. 

For example, consider the online optimization algorithm of~\citet{mcmahan2014unconstrained}, which requires an upper bound $L$ on the norm of stochastic gradients but makes only a weak assumption on how far the initialization $x_0$ (say, the origin) is from the optimum $x^\star$: the algorithm has an additional parameter $r_\eps$ which may viewed as a very conservative lower bound on $R=\norm{\xopt-x_0}$. The $r_\eps$ parameter may simply be set to a very small number (say, $10^{-8}$) and is not meant to be tuned. Applying the algorithm to stochastic convex optimization (SCO) via online-to-batch conversion, after $T$ steps we are guaranteed an expected optimality gap bound of $
O\prn[\big]{ \frac{L(R+r_\eps)}{\sqrt{T}} {
\sqrt{\log\prn*{({R+r_\eps})/{r_\eps}} }}}
$. To equivalently state this guarantee, consider the class $\Iclass[\Lip]^{L,R}$ of SCO problems with $L$-Lipschitz sample functions and solution norm at most $R$. Then, for any problem in $\Iclass[\Lip]^{L,R}$ such that $r_\eps \le R$, the algorithm  \cite{mcmahan2014unconstrained} guarantees error at most a factor $O(\sqrt{\log ({R}/{r_\eps})})$ larger than $\frac{LR}{\sqrt{T}}$, the minimax optimal error for the class $\Iclass[\Lip]^{L,R}$.%

We propose a more concise way of stating adaptivity guarantees such as the above, allowing us to argue about their optimality. Consider a large collection of problem classes, which we call a \emph{meta-class} and denote by $\Mclass$. For each problem class $\Iclass\in\Mclass$ we quantify the competitive ratio between an adaptive algorithm's (worst-case) suboptimality, and the best possible suboptimality when the class is known in advance. Inspired by the notion of ``price of anarchy'' in algorithmic game theory \cite{roughgarden2005selfish}, we define the \emph{price of adaptivity} (PoA) to be the highest competitive ratio incurred by any element in the meta-class. 

In the proposed framework, the adaptivity guarantee of \citet{mcmahan2014unconstrained} fits into one sentence: for any $\rho>2$ and the meta-class $\Mclass[\Lip]^{1, \rho} = \crl[\big]{\Iclass[\Lip]^{1,r}}_{r\in[1, \rho]}$ the PoA (with respect to expected suboptimality) is $O(\sqrt{\log \rho})$. Such logarithmic dependence on $\rho$ indicates strong robustness to uncertainty in the distance between the initial point and the optimum (a PoA of $1$ corresponds to perfect robustness). However, for the larger meta-class $\Mclass[\Lip]^{\ell, \rho} = \crl[\big]{\Iclass[\Lip]^{l,r}}_{l\in[1,\ell], r\in[1, \rho]}$ which also captures uncertainty in the Lipschitz constant, the PoA of~\cite{mcmahan2014unconstrained} is $O(\ell\sqrt{\log \rho})$. Such polynomial dependence on $\ell$ indicates sensitivity to uncertainty in the problem's Lipschitz constants.

\newcommand{\lbcolor}[1]{\cellcolor[HTML]{EFEFEF} \hspace{-2pt}\bf\boldmath\color{black}{#1}}

\begin{table}[b!]
	\renewcommand{\arraystretch}{1.15}
	\centering	
	\setlength{\tabcolsep}{0pt}
	\begin{tabular}{@{}llll@{}}
	\toprule
	Meta-class ~~~~~& Suboptimality ~~~~~~~& Algorithm  / Theorem \hspace{1.5cm}~& Price of Adaptivity \\ \midrule
	\multirow{4}{*}{$\Mclass[\Lip]^{1, \rho}$} & \multirow{2}{*}{expected} & \citet{mcmahan2014unconstrained} & $\O{\sqrt{\log \rho}}$ \\
	 &  &  \lbcolor{\Cref{thm:log-PoA-lb}} &  \lbcolor{$\Omeg{\sqrt{\log \rho}}$} \\ \cmidrule(l){2-4} 
	 & \multirow{2}{*}{$1-\delta$ quantile} & \citet{carmon2022making} & $\O{ \log\prn*{\frac{\log(\rho T)}{\delta}} }$ \\ %
	 &  & \lbcolor{\Cref{thm:loglog-PoA-lb}} & {\lbcolor{$\Omeg{\sqrt{\log \prn*{\frac{\log\rho}{\delta} }/\log\frac{1}{\delta}}}$}}~~$\dagger$~ \\ \midrule
	\multirow{4}{*}{$\Mclass[\Lip]^{\ell, \rho}$} & \multirow{2}{*}{expected} & \citet{cutkosky2018black} & $\O{\sqrt{\log (\rho)} + \frac{\ell \log(\rho)}{\sqrt{T}}}$ \\
	 &  & \citet{cutkosky2019artificial} & $\Otilb{\frac{\rho}{\sqrt{T}}}$ \\ %
	\cmidrule(l){2-4} 
	 & \multirow{2}{*}{$1-\delta$ quantile} & \citet{carmon2022making} & $\O{ \prn*{1+\frac{\ell}{\sqrt{T}}}\log^2\prn*{\frac{\log(\rho T)}{\delta}} }$ \\ %
	 &  & \lbcolor{\Cref{thm:poly-PoA-lb}} & \lbcolor{$\Omeg{\frac{\min\crl{\rho,\ell}}{\sqrt{T}}\sqrt{\log\frac{1}{\delta}}}$} \\ \midrule
	 \multirow{4}{*}{$\Mclass[\SMLip]^{\ell,\rho}$} & \multirow{2}{*}{expected} & SGD & $\O{\sqrt{\rho\ell}}$ \\
	 &  & Adaptive SGD~\cite{mcmahan2017survey,gupta2017unified} & $\O{\rho}$ \\ \cmidrule(l){2-4} 
	 & \multirow{2}{*}{$1-\delta$ quantile} & \citet{zhang2022parameter} & $\Otilb{\ell}$ \\ %
	 &  & \lbcolor{\Cref{thm:poly-PoA-lb}} & \lbcolor{$\Omeg{\min\crl{\rho,\ell}}$} \\ \bottomrule
	\end{tabular}
	\caption{A summary of lower and upper bounds on the price of adaptivity. Here $\ell$ is the uncertainty in the stochastic gradient bound (probability 1 for $\Mclass[\Lip]^{\ell, \rho}$ or second-moment for $\Mclass[\SMLip]^{\ell, \rho}$) and $\rho$ is the uncertainty in the initial distance to optimality. See \Cref{sec:sco-setup} for formal definitions and \Cref{app:upper-bound-deriv} for translation of published result to PoA bounds. In the setting we consider, the PoA is at least $1$ and at most $O(\sqrt{T})$; for brevity, we omit the minimum with $\sqrt{T}$ and maximum with $1$ in the lower and upper bounds, respectively. %
	The $\Otil{\cdot}$ notation hides poly-logarithmic factors. 
	$^\dagger$~This lower bound holds for stochastic first-order algorithms, while the rest hold for all stochastic optimization algorithms (see \Cref{sec:discussion}). 
	}
	\label{tab:poa-upper-lower-bounds}
\end{table}

\Cref{tab:poa-upper-lower-bounds} surveys---in terms of PoA---the state-of-the-art algorithms for SCO with unknown initial distance to optimality and/or stochastic gradient (i.e., Lipschitz) bound, corresponding to the meta-class $\Mclass[\Lip]^{\ell, \rho}$ described above. We also consider the meta-class $\Mclass[\SMLip]^{\ell,\rho}$ corresponding to second-moment Lipschitz problems, where the stochastic gradient bound applies only to their second moment (instead of applying with probability 1), and we analyze two error measures: expected and quantile. Notably, for $\rho = O(1)$ it is possible to be completely adaptive to the Lipschitz constant (PoA independent of $\ell$), 
and for $\ell = O(\sqrt{T})$ it is possible to obtain PoA logarithmic in $\rho$. However, even when $\ell=1$ all known PoA bounds have some dependence on $\rho$. Moreover, no algorithm has sub-polynomial PoA in both $\ell$ and $\rho$. Are these shortcomings evidence of a fundamental cost of not knowing the problem parameters in advance?

In this work we answer this question in the affirmative, proving three PoA lower bounds highlighted in \Cref{tab:poa-upper-lower-bounds}. 
Our \emph{first lower bound} shows that the expected-suboptimality PoA is always logarithmic in $\rho$, establishing that~\cite{mcmahan2014unconstrained} is optimal when the Lipschitz constant is known. Taken together with the double-logarithmic high-probability PoA upper bound \cite{carmon2022making}, our result has the counter-intuitive implication that parameter-free high probability bounds are fundamentally better than in-expectation bounds; see additional discussion in \Cref{sec:discussion}.
Our \emph{second lower bound} shows that for constant-probability suboptimality, double-logarithmic dependence on $\rho$ is unavoidable, establishing near-optimality of~\citep{carmon2022making}. 
Finally, our \emph{third lower bound} shows that sub-polynomial PoA in both $\ell$ and $\rho$ is impossible. For second-moment Lipschitz problems, we give a lower bound of $\Omega(\min\{\ell,\rho\})$. This proves that any robustness to distance comes at a cost of sensitivity to the Lipschitz constant, and vice versa, and establishes that a combination of adaptive SGD and the method of~\citet{zhang2022parameter} is nearly optimal. 
For problems with bounded stochastic gradients, our lower bound is smaller a factor of $\sqrt{T}$, but is still nearly tight, with a combination of \citet{carmon2022making} and \citet{cutkosky2019artificial} providing nearly-matching upper bounds. Altogether, these results provide a nearly complete picture of the price of adaptivity in non-smooth stochastic convex optimization.

As a side result, we prove matching upper and lower bounds on the minimax suboptimality quantiles for Lipschitz and second-moment Lipschitz problems, respectively. In other words, for \emph{known} problem parameters, whether the noise is light-tailed or heavy-tailed (with bounded second moment) has virtually no effect on the difficulty of obtaining high-probability guarantees in non-smooth stochastic convex optimization. Despite being a very basic result, we could not locate it in the literature. We leverage this result to prove our lower bounds on the high-probability PoA.

\paragraph{Paper organization and suggested reading order.} The remainder of the introduction surveys our proof techniques and related work. 
\Cref{app:info-bounds} establishes key lemmas on information-theoretic hardness that are useful for our results.
\Cref{sec:known-params} sets up the necessary concepts and notation for stochastic convex optimization with known parameters, and proves our matching upper and lower bounds on quantile minimax error.
\Cref{sec:define-PoA} formally defines the price of adaptivity, \Cref{sec:lbs} proves our lower bounds on it, and \Cref{sec:discussion} provides further discussion. 
For readers wishing to focus on the key aspects of our work we suggest skipping \Cref{app:info-bounds,app:high-prob-ms-lip,app:high-prop-lip-lb} in the first reading.

\paragraph{Notation.}
Throughout the paper, $\xset$ denotes a closed convex set (our lower bounds use $\xset=\R$) and $\| \cdot \|$ denotes the Euclidean norm of a vector, and the infimum Euclidean norm in a set of vectors (or infinity if the set is empty).
For convex function $h:\xset \to \R$ we write (with minor abuse of notation) $\grad h(x)$ for an arbitrary subgradient of $h$ at $x$, and let $\del h(x)$ denote the set of all subgradients at $x$. 
We say that $h$ is $L$-Lipschitz if $\abs{ h(x) - h(y) } \le L \| x - y \|$ for all $u, v \in \xset$ or, equivalently, $\| \grad h(x) \| \le L$ for almost-all $x\in\xset$. We use $[n]\defeq\{1,\ldots,n\}$ and $\indic{\cdot}$ for the indicator function.
Throughout, $\log(\cdot)$ denotes the natural logarithm and $\N$ starts from 1.

\subsection{Overview of proof techniques}\label{sec:overview-proof-techniques}

Formally defining the PoA allows us to leverage a primary technique for proving lower bounds in stochastic optimization: reduction from optimization to hypothesis testing~\cite{agarwal2012information,braun2017lower,duchi2018introductory}. That is, we carefully construct subsets of $\Mclass$ such that, on the one hand, it is information-theoretically impossible to reliably distinguish elements of these subsets, but on the other hand failing to distinguish elements incurs higher suboptimality than what is achievable with known problem parameters. The PoA perspective adds a new twist to this old technique by altering the correspondence between optimization error to statistical risk, requiring novel analyses and constructions. Below, we briefly describe how each of our three constructions (which are all one-dimensional) embed statistical problems. We then discuss our technique for establishing information-theoretic hardness. 

\paragraph{\Cref{thm:log-PoA-lb}.} To prove our $\Omega(\sqrt{\log \rho})$ lower bound on expected-suboptimality PoA, we apply a reduction to testing the bias of a coin.
That is, we construct problems where finding the optimum to sufficient accuracy requires correctly deciding whether $T$ coin flip results came from a coin biased by $+\eps/2$ or $-\eps/2$ away from the uniform distribution, for a parameter $\eps$ of our choosing. 
When the bias is positive, the minimizer is distance $\rho$ away from the initial point, and when the bias is negative its distance is below $1$. 
This construction lower bounds the PoA by a term proportional to $\eps \sqrt{T} (\rho \cdot p_{+|-} + p_{-|+})$, where $p_{+|-}$ is the probability of mistaking a negative bias for a positive one, and analogously for $p_{-|+}$. 
Using a careful information-theoretic argument, we show that the weighted error $\rho \cdot p_{+|-} + p_{-|+} = \Omega(1)$ for $\eps = \Omega(\sqrt{\log (\rho) / T})$, establishing the claimed lower bound. 
By contrast, in the classical SCO lower bound the excess suboptimality is proportional to $\eps \sqrt{T} (p_{+|-} + p_{-|+})$, and the unweighted error satisfies $p_{+|-} + p_{-|+} = \Omega(1)$ only for $\varepsilon = O(1/\sqrt{T})$. 

\paragraph{\Cref{thm:loglog-PoA-lb}.}
Our proof of the $\Omega(\sqrt{\log \log \rho})$ lower bound on constant-probability PoA is a reduction to the noisy binary search problem~\cite{karp2007noisy}: consider $n$ coins such that coin $i$ has bias $-\eps/2$ for $i\le i^\star$ and bias $\eps/2$ for $i> i^\star$, and for $T$ steps we get to choose a coin, flip it, and observe the outcome. We let $n:=\ceil{\log \rho}$ and construct a set of optimization problems such that the PoA is proportional to $\eps\sqrt{T}\P(\hat{i} \ne i^\star)$, where $\hat{i}$ is any estimator of $i^\star$ given the observed coin flips, and $i^\star$ is drawn uniformly from $[n]$. We provide a short self-contained proof of a lower bound from \citet{karp2007noisy} showing that $\P(\hat{i} \ne i^\star) = \Omega(1)$ for $\eps = \Omega( \sqrt{\log (n)/T})$ yielding our double-logarithmic lower bound (as $n=\Omega(\log\rho)$). To our knowledge, this is the first stochastic optimization lower bound to leverage the hardness of noisy binary search. Notably, this lower bound holds for \emph{stochastic first-order algorithms} that only access each sample function by computing a single gradient. In contrast, the other lower bounds we prove hold for the broader class of \emph{stochastic optimization algorithms}, that are allowed unrestricted access to each sample function.

\paragraph{\Cref{thm:poly-PoA-lb}.}
Finally, to prove our $\Omega(\min\crl{\rho,\ell}/\sqrt{T})$ and $\Omega(\min\crl{\rho,\ell})$ lower bounds for probability 1 and second-moment bounded stochastic gradients, respectively, we hide the uncertain parameters in a rare event. To do so, we construct two problem instances such that every point $x$ has an excess suboptimality factor of at least $\rho/\sqrt{T}$ for one function or $\ell/\sqrt{T}$ for the other function. Moreover, with constant probability both functions produce exactly the same observations, forcing any algorithm to pay an $\Omega(\min\crl{\rho,\ell}/\sqrt{T})$ PoA. The first problem instance we construct is simply $\alpha |x-\rho|$ with probability 1, for some $\alpha \le 1$. For the second instance, there is a $\lambda/T$ probability of instead observing $\frac{2 \alpha T}{\lambda} \abs{x}$ where $\lambda = \Theta(\log\frac{1}{\delta})$ guarantees that with probability at least $\delta$ we only observe $\alpha |x-\rho|$. Carefully choosing $\alpha$ to satisfy the Lipschitz or second-moment Lipschitz constraints yields the claimed lower bounds.

\paragraph{Information-theoretic hardness via bit-counting (\Cref{app:info-bounds}).} 
To prove \Cref{thm:log-PoA-lb}, we need to establish that the weighted error $\rho \cdot p_{+|-} + p_{-|+} = \Omega(1)$, with $p_{+|-}$ and $p_{-|+}$ defined above. 
Writing $\delta = 1/(1+\rho)$, this is equivalent to showing $p_{\text{err}}=(1-\delta)p_{+|-} + \delta p_{-|+} = \Omega(\delta)$, where $p_{\text{err}}$ is the probability of mis-estimating the coin's bias when it is drawn from a $\bernoulli(\delta)$ distribution. Our crucial observation is that each observation can reveal at most $O(\delta \eps^2)$ bits of information on the bias (formally, this is an elementary mutual information bound). Intuitively, as long as the accumulated information $O(\delta \eps^2 T)$ does not reach the entropy of the bias, $H(\bernoulli(\delta))\ge \delta \log\frac{1}{\delta}$, we cannot hope to estimate the bias with nontrivial accuracy; we formulate this intuition using Fano's inequality~\cite{fano1961transmission}. Thus, we may take $\eps = \Omega\prn[\Big]{\sqrt{\frac{H(\bernoulli(\delta))}{\delta T}}}\ge \Omega\prn[\Big]{\sqrt{\frac{\log({1}/{\delta})}{T}}}$ while still ensuring $p_{\text{err}}=\Omega(\delta)$ as required. As it happens, the same argument also proves our lower bound for the minimax quantile error for known problem parameters (\Cref{prop:high-prop-lip-lb}). 

Our proof of the noisy binary search hardness at the heart of \Cref{thm:loglog-PoA-lb} follows along similar lines. There, each observation yields $O(\eps^2)$ bits of information, but the entropy of the estimand $i^\star$ is $\log n$, and therefore we intuitively expect non-trivial error probability as long as $\eps^2 T \le O(\log n)$; Fano's inequality readily confirms this intuition. Thus, we may take $\eps = \Omega\prn[\big]{\sqrt{{(\log n)}/{T}}}$ as required.

\subsection{Related work}\label{sec:related-work}

We now discuss the price of adaptivity in two settings adjacent to ours: online convex optimization and non-stochastic optimization. 

\paragraph{Online convex optimization.} A long line of work starting from~\citet{streeter2012no} studies \emph{parameter-free} online optimization methods that require little or no tuning~\cite[see, e.g.,][]{orabona2014simultaneous,orabona2016coin,mcmahan2017survey,zhang2022pde,chen2022better,orabona2021parameter,cutkosky2018black,cutkosky2019artificial,mhammedi2019lipschitz,sharrock2023coin,jacobsen2023unconstrained}. These methods directly imply PoA upper bounds via online-to-batch conversion, and this is how many of the upper bounds in \Cref{tab:poa-upper-lower-bounds} are derived (see \Cref{app:upper-bound-deriv}). Conversely, our PoA lower bounds imply regret lower bounds for parameter-free online algorithms.

The online parameter-free literature also has a number of lower bounds, which do not have immediate implications on stochastic convex optimization. Nevertheless, Theorem 2 of \citet{orabona2013dimension} is analogous to our $\Omega(\sqrt{\log \rho})$ lower bound (\Cref{thm:log-PoA-lb}), showing that for every online algorithm and $\rho>0$ there exists $u$ with $\norm{u}=\rho$ and a sequence of vectors with norm 1 for which $u$ incurs regret $\Omega\prn*{ \| u \| \sqrt{T \log \rho} }$. 
In terms of proof techniques the two lower bounds appear to be fairly different, with~\cite{orabona2013dimension} using the probabilistic method,  uniformly drawn Rademacher variables, and a careful tail bound, whereas we perform a reduction to hypothesis testing and appeal to information-theoretic hardness.
Other online parameter-free lower bounds~\cite{cutkosky2017online,mhammedi2020lipschitz} are not directly comparable to our results, since it is not clear how to translate them into statements on competitive ratios. 

\paragraph{Non-stochastic optimization.}
In \emph{smooth} optimization with exact gradients, adaptivity can come almost for free. In particular, line search techniques such as~\cite{beck2009fast} and~\cite{carmon2022optimal} converge with optimal complexity up to an \emph{additive} logarithmic factor, implying PoA that tends to 1 as $T$ grows. In the non-smooth setting, Polyak's method \cite{polyak1987} attains optimal rates if we assume function value access and knowledge of the optimum function value. With only a lower bound on the optimum value, a variant of Polyak's method attains logarithmic PoA~\cite{hazan2019revisiting}. Without further assumptions, however, the best PoA bound in the non-smooth setting is $\O{\sqrt{\log \log (\rho T)}}$, which follows from setting $\eta_{\epsilon} = \| g_0 \|^{-1} T^{-1}$ in \cite[Theorem 7]{carmon2022making} and returning the algorithm's output or the initial point if the latter has lower objective value.

We are not aware of general PoA lower bounds in the non-stochastic settings. \citet[Section 4]{mishchenko2023prodigy} give a lower bound on a restricted class of algorithms (excluding~\cite{carmon2022making} for example), that only applies for  $T=O(\log \rho)$, where it holds essentially by definition of the algorithm class.

\paragraph{Concurrent work.} In concurrent and independent work, \citet{attia2024free} and \citet{khaled2024tuning} also study limitations imposed by problem parameter uncertainty in stochastic optimization. Both works consider whether an algorithm can match the rate of tuned SGD up to polylogarithmic factors despite uncertainty in all problem parameters; a property \citet{khaled2024tuning} call ``tuning-free.'' Stated in our terminology, this is asking whether the PoA when competing against the class of SGD algorithms is polylogarithmic. For convex stochastic optimization, they show that tuning-free algorithms generally do not exist by proving lower bounds similar to the first part \Cref{thm:poly-PoA-lb} (they do not consider the heavy-tailed setting). The papers \cite{attia2024free,khaled2024tuning} do not attempt to characterize precise logarithmic factors as we do in \Cref{thm:log-PoA-lb,thm:loglog-PoA-lb}. Instead, they consider broader settings including smooth and bounded-variance stochastic optimization both with and without convexity, and provide tuning-free algorithms in a number of settings. Moreover, the two works propose different assumptions under which tuning-free methods exist for stochastic convex optimization.

\section{Information-theoretic hardness lemmas}\label{app:info-bounds}

In this section, we establish two information-theoretic hardness results that facilitate the proofs of \Cref{thm:log-PoA-lb,thm:loglog-PoA-lb,prop:high-prop-lip-lb}. In each result, we show it is difficult to estimate a random quantity $V$ from a sequence of binary observations $S_1, \ldots, S_T$ that only carry little information about $V$, in the following sense: conditional on $V$ and $S_1, \ldots, S_{t-1}$, the distribution of $S_t$ is at most $\eps$ far away from a uniform Bernoulli.

We begin with a general bound on the mutual information between $V$ and any estimator $\hat{V}$ of $V$ based on $S_1, \ldots, S_T$ (\Cref{app:general-mi-bound}). 
We then consider the case that $V\sim \bernoulli(p)$ and show that any estimator has error probability $\Omega(p)$ for 
 $T=\Omega\prn[\big]{\frac{\log({1}/{p})}{\eps^2}}$ (\Cref{app:skewed-binary-bound}). Finally, we consider $V\sim\uniform([n])$ and show that any estimator $\hat{V}$ (even a randomized one capable of adaptively influencing the $S_1, \ldots, S_T$) has constant error probability for  $T=\Omega\prn[\big]{\frac{\log n}{\eps^2}}$ (\Cref{app:noisy-binary-search-bound}). 

The developments in this section use the standard notation
\[
    I(X;Y) = H(X)-H(X\mid Y)=H(Y)-H(Y\mid X)
\]
for mutual information $I(X;Y)$, entropy $H(X)=\E_{x\sim X}\log \frac{1}{\Pr(X=x)}$ and conditional entropy $H(X\mid Y) = \E_{x,y\sim X,Y}\log \frac{1}{\Pr(X=x | Y=y)}$.  We also write
\[
    h_2(q) \defeq H(\bernoulli(q)) = q \log\frac{1}{q} + (1-q) \log\frac{1}{1-q}    
\]
for the binary entropy function.
Our bounds hinge on the following classical result.
\begin{fact}[Fano's inequality~\cite{fano1961transmission}] \label{fact:fano}
Let $V$ be a random variable taking $n$ values, and let $\hat{V}$ be some estimator of $V$. Then
\[
h_2(\Pr(V\ne \hat{V})) + \Pr(V\ne \hat{V})\log(n-1) \ge H(V\mid \hat{V}).
\]
\end{fact}

\noindent
We remark that prior work \citep{agarwal2012information,braun2017lower} also provide optimization lower bounds by controlling mutual information and applying Fano's inequality, but in different settings.

\subsection{A general mutual information bound}\label{app:general-mi-bound}
Consistent with the rest of this paper, the following lemma measures entropy and mutual information in nats (i.e., the logarithm basis is $e$). 

\begin{lem}\label{lem:general-mi-bound}
    For any $T\in \N$ and $\eps \in [0,1]$ and any random variable $V$, let $S_1,\ldots, S_T$ be a sequence of binary random variables such that, for all $t\le T$ \yair{could also write it as a TV distance bound}
    \[
        \abs*{\Pr(S_t = 0| S_1, \ldots, S_{t-1}, V) - \half} \le \frac{\eps}{2}
    \]
     with probability 1 w.r.t.\ $S_1,\ldots, S_{t-1}$ and $V$. Then, for every randomized estimator $\hat{V}$ that depends on $V$ only through $S_1, \ldots, S_T$, we have
    \[
        I(V;\hat{V}) \le \eps^2 T.\numberthis \label{eq:general-mi-bound-1}
    \]
    If in addition $\eps \le \half$ and, for some $p\in [0,\half]$, we have that
    \[
        \abs*{\Pr(S_t = 0) - \half} \ge \eps \prn*{\frac{1}{2} - p}
    \]
    for every $t\le T$, then
    \[
        I(V;\hat{V}) \le 4p \eps^2 T.\numberthis \label{eq:general-mi-bound-2}
    \]
\end{lem}

\begin{proof}
    By the data processing inequality \citep[Theorem 2.8.1]{cover1991elements} and chain rule for mutual information \cite[Theorem 2.5.2]{cover1991elements}, 
	\begin{flalign*}
		I(\hat{V}; V) & \le I(S_1, \ldots, S_T; V) 
		= \sum_{t=1}^{T} \brk[\Big]{ H(S_t \mid S_1, \ldots, S_{t-1}) - H(S_t \mid S_1, \ldots, S_{t-1}, V) }. \numberthis \label{eq:log-mi-expression}
	\end{flalign*}
	Using the facts conditioning decreases entropy and that binary random variables have entropy at most $\log 2$, we get
    \[
		H(S_t \mid S_1, \ldots, S_{t-1}) \le H(S_t) \le \log 2.
	\]
    Next, we have 
    \[
		H(S_t \mid S_1, \ldots, S_{t-1}, V) = \E_{S_1, \ldots, S_{t-1}, V} \brk*{h_2 \prn[\Big]{\Pr(S_t = 0| S_1, \ldots, S_{t-1}, V) }}
         \ge  h_2\prn*{\frac{1-\eps}{2}},
    \]
    by the assumption that $\abs*{\Pr(S_t = 0| S_1, \ldots, S_{t-1}, V) - \half} \le \frac{\eps}{2}$ with probability 1 and the fact that $h_2(x)$ has a single maximum at $x=\half$.

    Combining the last two displays and using $\log(1+x)\le x$, we obtain
    \begin{flalign*}
        H(S_t \mid S_1, \ldots, S_{t-1}) &- H(S_t \mid S_1, \ldots, S_{t-1}, V) \\& 
        \le \log 2 - h_2\prn*{\frac{1-\eps}{2}}
        \\ &=
		\frac{1+\veps}{2} \log(1+\veps) + \frac{1-\veps}{2} \log(1-\veps) 
        \le 
		\veps^2,
    \end{flalign*}
    where the last inequality is due to $\log(1+x)\le x$.
    Substituting back into~\cref{eq:log-mi-expression} yields the first claimed bound, $I(\hat{V}; V) \le \eps^2 T$. 

    To obtain the second claimed bound we use the assumption that $\abs*{\Pr(S_t = 0) - \half} \ge \eps \prn*{\frac{1}{2} - p}$ to write
    \[
		H(S_t \mid S_1, \ldots, S_{t-1}) \le H(S_t) 
        \le h_2\prn*{\frac{1-(1-2p)\eps}{2}},
    \]
    improving the bound on $H(S_t \mid S_1, \ldots, S_{t-1}) - H(S_t \mid S_1, \ldots, S_{t-1}, V)$ to 
    \begin{flalign*}
        h_2\prn*{\frac{1-(1-2p)\eps}{2}} - h_2\prn*{\frac{1-\eps}{2}}
         & \overle{(i)} 
        \prn*{\frac{1-(1-2p)\eps}{2} - \frac{1-\eps}{2}} h_2'\prn*{\frac{1-\eps}{2}}
        \\ & \overeq{(ii)} 
        p \eps \log\prn*{1+\frac{2\eps}{1-\eps}} \overle{(iii)} \frac{2p\eps^2}{1-\eps} \overle{(iv)} 4p\eps^2,
    \end{flalign*}
    due to $(i)$ concavity of binary entropy, $(ii)$ the fact that $h_2'(x) = \log\prn*{\frac{1-x}{x}}$, $(iii)$ the inequality $\log(1+x)\le x$ and $(iv)$ our assumption that $\veps \le \half$. Substituting back into~\eqref{eq:log-mi-expression}, we get that $I(\hat{V}; V) \le 4p\eps^2 T$, as required.
\end{proof}

\subsection{Hardness of sharpening a biased Bernoulli prior}\label{app:skewed-binary-bound}

\begin{lem}\label{lem:skewed-binary-bound}
	For any $p\in(0,\frac{1}{2}]$ and $T\in \N$, let $V\sim \bernoulli(p)$ and
	\[
	\eps = \min\crl*{\sqrt{\frac{\log \frac{1-p}{p}}{8T}}, \half}.	
	\]
	Let $P_0 = \bernoulli(\frac{1-\eps}{2})$ and $P_1 = \bernoulli(\frac{1+\eps}{2})$. Then, for $S_1, \ldots, S_T \simiid P_V$ and (potentially randomized) estimator $\hat{V}$ that depends on $V$ only through $S_1, \ldots, S_T$, we have %
	\[
		\Pr(\hat{V} \ne V) \ge \frac{p}{2}.
	\]
\end{lem}
\begin{proof}
	To streamline notation, we write 
	\[
		q \defeq \Pr(\hat{V} \ne V). 
	\] 
	By Fano's inequality (\Cref{fact:fano}) with $n=2$,
	\[
		h_2(q) \ge H(V\mid \hat{V}) =  H(V) - I(V;\hat{V}).	
	\numberthis \label{eq:skewed-fano}
	\]
	by \cref{eq:general-mi-bound-2} in \Cref{lem:general-mi-bound}, 
    \[
		I(\hat{V}; V) \le 4p\eps^2 T \overle{(i)} \frac{p}{2} \log\frac{1-p}{p} \overeq{(ii)} \prn*{p - \frac{p}{2}}h_2'(p) \overle{(iii)} h_2(p) - h_2(p/2), 
	\]
	due to $(i)$ our choice of $\eps \le \sqrt{\frac{\log \frac{1-p}{p}}{8T}}$, (ii) the fact that $h_2'(x) = \log\frac{1-x}{x}$ and $(iii)$ concavity of entropy. Noting that $H(V) = h_2(p)$ and substituting back to~\eqref{eq:skewed-fano}, we get that $h_2(q) \ge h_2(p/2)$, which implies $q\ge p/2$ since $h_2$ is monotone in $[0, \half]$. 
\end{proof}

\subsection{Hardness of estimating a uniform index}\label{app:noisy-binary-search-bound}
\begin{lem}\label{lem:noisy-binary-search-bound}
    For any $n > 16$ and $T\in \N$, let $V\sim \uniform([n])$ and
	\[
	\eps = \min\crl*{\sqrt{\frac{\log n}{4T}}, 1}.	
	\]
	Let $S_1, \ldots, S_T$ be binary random variables  
	such that 
    $\abs*{\Pr(S_t = 0| S_1, \ldots, S_{t-1}, V) - \half} \le \frac{\eps}{2}$ for all $t\le T$ with probability 1. 
    Then, for any (potentially randomized) estimator $\hat{V}$ that depends on $V$ only through $S_1, \ldots, S_T$, where have %
	\[
		\Pr(\hat{V} \ne V) > \frac{1}{2}.
	\]
\end{lem}
\begin{proof}
    We have 
    \[
        H(V\mid \hat{V}) = H(V) - I(V;\hat{V}) = \log n - I(V;\hat{V}),
    \]
    with the second equality due to the uniform distribution of $V$. Substituting into Fano's inequality (\Cref{fact:fano}), rearranging, and using $h_2(q) \le \log 2$ for all $q$, gives the following well-known corollary:
    \begin{equation}\label{eq:not-really-fano's-inequality}
        \Pr(\hat{V} \ne V) \ge \frac{\log(n) - I(V;\hat{V}) - h_2\prn*{\Pr(\hat{V} \ne V) }}{\log(n-1)} \ge 1 - \frac{I(V;\hat{V}) + \log(2)}{\log n}.
    \end{equation}
    The bound~\eqref{eq:general-mi-bound-1} in \Cref{lem:general-mi-bound} and our choice of $\eps$ show that 
    \begin{equation*}
    I(\hat{V}; V) \le  \veps^2 T \le \frac{1}{4} \log{n}.
    \end{equation*}
	Substituting back into \cref{eq:not-really-fano's-inequality} and using $n> 16$ we obtain
	\begin{equation*}
		\Pr(\hat{V} \ne V) > \frac{1}{2}
	\end{equation*}
	as required.
\end{proof}

\section{Stochastic convex optimization with known parameters}\label{sec:known-params}

\subsection{Formal setup}\label{sec:sco-setup}
This section formally defines the basic building blocks of our paper: 
stochastic optimization problems, algorithms, error metrics, and minimax rates.

\paragraph{Stochastic optimization (SO) problems.}
A SO problem instance is a tuple $(f,P)$ containing a distribution $P$ over sample space $\sset$ and a sample objective $f:\xset\times \sset \to \R$. We write
\[
F_{f,P}(x) \defeq \E_{S\sim P} f(x;S) 
~~\mbox{and}~~
X^\star_{f,P} \defeq \argmin_{x\in\xset} F_{f,P}(x)
\]
for the population objective (which we wish to minimize) and its set of minimizers, respectively. Let $\Iclass$ denote a class of SO problem instances. We consider two fundamental classes of convex functions with a minimizer at most $R$ away from the origin.\footnote{This implicitly, and without loss of generality, assumes that optimization methods are initialized at $x_0=0$.} The first class contains problems with bounded stochastic gradient norm (i.e., where each sample function is $L$-Lipschitz),
\begin{equation*}
	\Iclass[\Lip]^{L,R} \defeq \crl*{(f,P) \mid f(\cdot; s) \text{ is convex and $L$-Lipschitz for all $s \in \sset$, and } \norm*{X^\star_{f,P}} \le R }.
\end{equation*}
The second class contains problems with bounded stochastic gradient second moment, 
\begin{equation*}
	\Iclass[\SMLip]^{L,R} \defeq 
	\crl*{(f,P) \mid f(\cdot; s) \text{ is convex $\forall s\in\sset$ and }
	\E_{S\sim P}\norm{\grad f(\cdot; S)}^2 \le L^2,\text{ and } \norm*{X^\star_{f,P}} \le R }.
\end{equation*}
Clearly, $\Iclass[\Lip]^{L,R} \subset \Iclass[\SMLip]^{L,R}$, with $\Iclass[\SMLip]^{L,R}$ also including problems with heavy-tailed gradient noise.\yair{consider defining $\Iclass[*\text{-}Lip]$ and $\Mclass[*\text{-}Lip]$ as a catch-both term.}

\paragraph*{Optimization algorithms.}
General SO algorithms have unrestricted access to the function $f(\cdot;\cdot)$ but observe $P$ only through samples $S_1,S_2,\ldots\simiid P$. The output of general SO algorithms is therefore of the form
\[
x_{t} = \alg_t\prn[\big]{ f(\cdot; S_1) , \dots,  f(\cdot; S_T) ; \xi } \in \xset,
\]
where $\alg_t$ is some measurable mapping and $\xi \sim \uniform([0,1])\perp S_1, \ldots, S_T$ allows randomization.

Stochastic first-order (SFO) optimization algorithms are SO algorithms that depend on $f$ only through gradients observed at past iterates. That is, their output sequence takes the form 
\[
x_{t} = \alg_{t}\prn*{\grad f(x_0; S_1) , \dots,  \grad f(x_{t-1}; S_t); \xi} \in \xset,
\]
with $\xi \sim \uniform([0,1])\perp S_1, \ldots, S_T$ as before.
We write $\Aall$ and $\AFO$ for the sets of all SO and all SFO algorithms, respectively, so that $\AFO\subset \Aall$. All the algorithms from \Cref{tab:poa-upper-lower-bounds} are in $\AFO$.

\paragraph{Error metrics.}
For an algorithm $\alg$, SO instance $(f,P)$, and budget $T$, we let $x_T$ denote the algorithm's output given $T$ samples and define the expected error to be
\begin{equation*}
	\ErrE{\alg}{(f,P)} \defeq \E F_{f,P}(x_T) - \inf_{\xopt\in\xset} F_{f,P}(\xopt),
\end{equation*}
where the expectation is with respect to $x_T$ and hence $S_1,\ldots,S_T$ and $\xi$. 
We similarly define the high-probability error at confidence level $\delta$ as
\begin{equation*}
	\ErrHP{\alg}{(f,P)} \defeq Q_{1-\delta}( F_{f,P}(x_T) ) - \inf_{\xopt\in\xset} F_{f,P}(\xopt),
\end{equation*}
where $Q_p(Y)=\min\{ y : \Pr( Y \le y ) \ge p\}$ denotes the $p$'th quantile of random variable $Y$ \cite[Ch. 3]{borovkov1999probability}.
Throughout, we drop the superscripts $\E$  and $\delta$ when making statements that apply to both.

\paragraph{Minimax error.}
Given an instance class $\Iclass$ and an algorithm $\alg$, we overload the notation above to denote worst-case error over the class,
\begin{equation*}
	\Err{\alg}{\Iclass} \defeq \sup_{(f,P)\in\Iclass} \Err{\alg}{(f,P)}.
\end{equation*}
Given an algorithm class $\Aclass$, we further overload our notation to express minimax optimal worst-case error
\begin{equation}
	\Err{\Aclass}{\Iclass} \defeq \inf_{\alg\in\Aclass} \Err{\alg}{\Iclass}
	= \inf_{\alg\in\Aclass} \sup_{(f,P)\in\Iclass} \Err{\alg}{(f,P)}.
	\label{eq:minimax-err-def}
\end{equation}

\Cref{prop:standard-error-bounds} characterizes the expected and quantile minimax error for the instance and algorithm classes defined above.

\begin{restatable}{prop}{restatePropStandardErrorBounds}\label{prop:standard-error-bounds}
For all $R, L > 0$, instance class $\Iclass \in \{ \Iclass[\Lip]^{L,R}, \Iclass[\SMLip]^{L,R} \}$ and 
$\Aclass \in \{ \Aall, \AFO \}$, we have 
$\ErrE{\Aclass}{\Iclass}=\Theta\left( \frac{LR}{\sqrt{T} } \right)$. Moreover, for any $\delta\in(0,\half)$ we have $\ErrHP{\Aclass}{\Iclass}=\Theta\left( \frac{LR}{\sqrt{1+ {T}/{\log \frac{1}{\delta}}}} \right)$.
\end{restatable}

\begin{proof}
Since $\Iclass[\Lip]^{L,R}\subset \Iclass[\SMLip]^{L,R}$ and $\AFO\subset\Aall$, we only need to prove minimax error lower bounds on $(\Aall,\Iclass[\Lip]^{L,R})$ and minimax upper bounds on $(\AFO, \Iclass[\SMLip]^{L,R})$. 

For expected error these results are well-known. The upper bound on $\ErrE{\AFO}{ \Iclass[\SMLip]^{L,R}}$ follows, for example, by setting $\eta = \frac{R}{L\sqrt{T}}$ in \cref{eq:sgd-expected-err}. For the lower bound on $\ErrE{\Aall}{ \Iclass[\Lip]^{L,R}}$, originally due to~\citet{nemirovski1983problem}, see, e.g., \citet[Theorem 5.2.10.]{duchi2018introductory}.

For error quantiles, however, we could not find direct proof of the required upper and lower bounds, and we prove them in \Cref{app:high-prob-ms-lip,app:high-prop-lip-lb} below.
\end{proof}

\subsection{Optimal high-probability bound for second-moment Lipschitz problems}\label{app:high-prob-ms-lip}

\newcommand{\gclip}{G_{\mathrm{clip}}}
\newcommand{\gbar}{\bar{g}}

In this section we derive high-probability guarantees for the clipped SGD method and the class $\Iclass[\SMLip]^{L,R}$ of second-moment $L$-Lipchitz stochastic convex optimization problems with initial distance to optimality $R$.
Our rate of convergence is the same (up to a constant) as the rate standard SGD achieves for the smaller class~$\Iclass[\Lip]^{L,R}$ of probability 1 $L$-Lipschitz functions.

A number of prior works~\cite[e.g.,][]{nazin2019algorithms,gorbunov2020stochastic,gorbunov2021near,davis2021low,sadiev2023high,nguyen2023improved} provide high probability bounds under heavy-tailed noise and many of them use gradient clipping. However, these works are mostly  concerned with the smooth setting. The only exception is~\citet{gorbunov2021near}, who consider non-smooth optimization under more general heavy-tailed noise distributions, but obtain rates that are suboptimal by polylogarithmic terms in the desired confidence level $\delta$.

The only probabilistic tool we require is the following simplification of Freedman's inequality.
\begin{fact}[Simplified form of Freedman's inequality \cite{freedman1975tail}] \label{fact:freedman}
    Let $X_1, \ldots, X_T$ be a sequence of real-valued random variables that satisfy $\Ex*{ X_t^2 | X_1, \ldots, X_{t-1} } \le A^2$ and $|X_t| \le B$, with probability $1$ for all $t\le T$. Then, for any $\delta \in (0,1)$, 
    \[
    \Pr*(
        \sum_{t=1}^T X_t > \sum_{t=1}^T \Ex*{X_t | X_1, \ldots, X_{t-1}} + A\sqrt{2 T \log \frac{1}{\delta}} + \frac{B}{3} \log \frac{1}{\delta}
    ) \le \delta.
    \]
\end{fact}

Given stochastic optimization instance ($f,P$), step size parameter $\eta$, clipping parameter $\gclip$ and initial point $x_0 \in \xset$, clipped SGD consists of the following recursion:
\begin{equation*}
    x_{t+1} = \Pi_{\xset'}\prn*{x_t - \eta \gbar_t},~~\mbox{where}~~\gbar_t \defeq \frac{g_t}{\max\{1, \norm{g_t} / \gclip\}}~~\mbox{and}~~g_t \in \del f(x_t; S_{t+1}).
    \numberthis \label{eq:clipped-SGD}
\end{equation*}
Here $\Pi_{\xset'}$ is the Euclidean projection onto $\xset' = \xset \cap \{ x \mid \norm{x-x_0} \le R\}$.

With these definitions in hand, we state and prove our rate of convergence guarantee.
\begin{customprop}{1a}\label{prop:ms-lip-hp-bound}
    For any $\delta \in (0,\half)$, any $L,R>0$, and any instance in $\Iclass[\SMLip]^{L, R}$, we have that $T$ steps of clipped SGD~\eqref{eq:clipped-SGD} with step size $\eta = \frac{R}{L\sqrt{T}}$ and $\gclip = \frac{L\sqrt{T}}{\sqrt{\log (1/\delta)}}$ produce iterates $x_0, \ldots, x_{T-1}$ whose average $\bar{x}_T = \frac{1}{T} \sum_{t=0}^{T-1} x_t$ has suboptimality $O\prn*{\frac{LR}{\sqrt{T}}\sqrt{\log\frac{1}{\delta}}}$ with probability at least $1-\delta$.
\end{customprop}

\begin{proof}
    Let $F(x) = \E_{S\sim P} f(x;S)$ be the objective function and let $\xopt$ be its minimizer in $\xset'$ (and hence, by assumption, also in $\xset$). The standard SGD regret bound, e.g., \cite[Lemma 2.12.]{orabona2021modern}, gives
    \[
    \sum_{t=0}^{T-1} \inner{\gbar_t}{x_t - \xopt} \le \frac{R^2}{2\eta} + \frac{\eta}{2} \sum_{t=0}^{T-1} \norm{\gbar_t}^2
    \]
    with probability 1. 
    Writing $\E_t Z = \Ex*{Z | x_0, \ldots, x_t}$ to streamline notation, 
    and using $\gbar_t = \E_t g_t - (\E_t g_t - \E_t \gbar_t) - (\E_t \gbar_t - \gbar_t)$ we rearrange the above display to read,
    \[
        \sum_{t=0}^{T-1} \inner{\E_t g_t}{x_t - \xopt} 
        \le \frac{R^2}{2\eta} 
        + \sum_{t=0}^{T-1} \inner{\E_t g_t - \E_t \gbar_t}{x_t - \xopt} 
        + \sum_{t=0}^{T-1} \inner{\E_t \gbar_t - \gbar_t}{x_t - \xopt}
        + \frac{\eta}{2} \sum_{t=0}^{T-1} \norm{\gbar_t}^2. 
        \numberthis \label{eq:clipped-sgd-regret-rearranged}
    \]
    
    Since $\E_t g_t\in\del F(x_t)$, we have $\inner{\E_t g_t}{x_t - \xopt} \ge F(x_t) - F(\xopt)$ and therefore, by Jensen's inequality we obtain the following lower bound on the LHS of \Cref{eq:clipped-sgd-regret-rearranged}
    \[
        \sum_{t=0}^{T-1} \inner{\E_t g_t}{x_t - \xopt}  \ge T \brk*{F(\bar{x}_T) - F(\xopt)}.
        \numberthis \label{eq:clipped-sgd-bound-1}
    \]
    We have $R^2 / (2 \eta) = O(LR\sqrt{T})$ by our choice of $\eta$, and it remains to bound the three sums in the RHS of \Cref{eq:clipped-sgd-regret-rearranged}.

    To bound the first sum, we observe that
    \begin{flalign*}
        \norm{\E_t g_t - \E_t \gbar_t} &\overle{(i)} \E_t \norm{g_t - \gbar_t} \overeq{(ii)} \E_t \brk*{\norm*{g_t} \prn*{1 - \frac{\gclip}{\| g_t \|} } \indic{\norm{g_t} > \gclip}} 
       \le 
        \E_t \brk*{\norm{g_t}\indic{\norm{g_t} > \gclip}}  \\ &
        \le \frac{1}{\gclip} \E_t \norm{g_t}^2 \overle{(iii)} \frac{L^2}{\gclip},
    \end{flalign*}
    where $(i)$ is from Jensen's inequality applied to the Euclidean norm, $(ii)$ from the definition of $\gbar_t$ and $(iii)$ is from the definition of $\Iclass[\SMLip]^{L, R}$. 
    Therefore, using the triangle inequality, Cauchy-Schwarz, the bound $\norm{x_t-\xopt}\le 2R$, and our setting of $\gclip$, we get that
    \[
    \abs*{\sum_{t=0}^{T-1} \inner{\E_t \gbar_t - \E_t g_t}{x_t - \xopt} } \le \sum_{t=0}^{T-1} \norm{\E_t g_t - \E_t \gbar_t} \norm{x_t - \xopt}  \le \frac{2L^2R T}{\gclip} = 2LR \sqrt{T\log\frac{1}{\delta}}.
    \numberthis \label{eq:clipped-sgd-bound-2}
    \]

    To bound the second sum, we note that $\E_t \brk*{\inner{\gbar_t - \E_t \gbar_t}{x_t - \xopt}} = 0$ for all $t$. Moreover, again by Cauchy-Schwarz and $\norm{x_t-\xopt}\le 2R$, we have that 
    \[
        \E_t \brk*{\inner{\gbar_t - \E_t \gbar_t}{x_t - \xopt}}^2
        \le 4R^2 \E_t \norm{\gbar_t - \E_t \gbar_t}^2 
        \le 4R^2 \E_t \norm{\gbar_t}^2 
        \le 4R^2 \E_t \norm{g_t}^2 
        \le 4L^2 R^2,
    \]
    as well as
    \[
        \abs*{\inner{\gbar_t - \E_t \gbar_t}{x_t - \xopt}}
        \le 
        \norm{\gbar_t - \E_t \gbar_t} \norm{x_t - \xopt}
        \le 4\gclip R,
    \]
    due to the gradient clipping.
    Therefore, we may apply \Cref{fact:freedman} on $X_t = \inner{\E_t\gbar_t -  \gbar_t}{x_t - \xopt}$ to conclude that, with probability at least $1-\frac{\delta}{2}$,
    \[
        \sum_{t=0}^{T-1} \inner{\E_t \gbar_t - \gbar_t}{x_t - \xopt}
        \le LR \sqrt{8T \log\frac{2}{\delta}} + \frac{4}{3} \gclip R \log\frac{2}{\delta} = O\prn*{LR \sqrt{T \log\frac{1}{\delta}} }.
        \numberthis \label{eq:clipped-sgd-bound-3}
    \]

    For the third and final sum, we observe that $\E_t \norm{\gbar_t}^4 \le \gclip^2 \E_t \norm{\gbar_t}^2 \le (\gclip L)^2$, and that $\norm{\gbar_t}^2 \le \gclip^2$ with probability 1. We then apply \Cref{fact:freedman} again, this time with $X_t = \norm{\gbar_t}^2$ to obtain that, with probability at least $1-\frac{\delta}{2}$,
    \begin{flalign*}
        \frac{\eta}{2} \sum_{t=0}^{T-1} \norm{\gbar_t}^2 
        & \le \frac{\eta}{2} \sum_{t=0}^{T-1} \E_t\norm{\gbar_t}^2 
        + \frac{\eta}{2} \gclip L \sqrt{ 2T \log \frac{1}{\delta}} 
        + \frac{\eta}{6} \gclip^2 \log \frac{1}{\delta}
        \\ &
        \overle{(i)} \frac{\eta}{2} L^2 T \cdot \prn*{1 + \sqrt{2} + \frac{1}{3} } \overle{(ii)} O\prn*{LR \sqrt{T }},
        \numberthis \label{eq:clipped-sgd-bound-4}
    \end{flalign*}
    due to $(i)$ the choice $\gclip$ and $\E_t\norm{\gbar_t}^2 \le \E_t\norm{g_t}^2 \le L^2$, and $(ii)$ the choice of $\eta$. 

    Substituting \cref{eq:clipped-sgd-bound-1,eq:clipped-sgd-bound-2,eq:clipped-sgd-bound-3,eq:clipped-sgd-bound-4} into \cref{eq:clipped-sgd-regret-rearranged} and using again the choice of $\eta$ and a union bound over the two events holding with probability at least $1-\frac{\delta}{2}$, we conclude that $T \brk*{F(\bar{x}_T) - F(\xopt)} \le O\prn*{LR \sqrt{T \log\frac{1}{\delta}} }$ with probability at least $1-\delta$, leading to the claimed rate of convergence.
\end{proof} %
\subsection{Optimal error quantile lower bound}\label{app:high-prop-lip-lb}

\begin{customprop}{1b}\label{prop:high-prop-lip-lb}
    For all $L,R>0$ and $\delta\in (0,\frac{1}{2})$ we have $\ErrHP{\Aall}{\Iclass[\Lip]^{L,R}} \ge 
    \Omega\prn*{  
        LR \min\crl*{ \sqrt{\frac{\log\frac{1}{\delta}}{T}}, 1 }
        }$.
\end{customprop}

\begin{proof}
    We prove \Cref{prop:high-prop-lip-lb} with a construction almost identical to the one in the proof of \Cref{thm:log-PoA-lb}. 
    We let $\xset=\R$, $\sset = \{0,1\}$ and $f:\xset\times\sset\to\R$ be
	\begin{equation*}
		f(x;s) = \begin{cases}
			L|x| & s = 0 \\
			L|x-R| & s = 1,\\
		\end{cases}
	\end{equation*}
	and we set, assuming for now $\delta \le \frac{1}{5}$,
	\begin{equation*}
		\eps \defeq \min\crl*{\sqrt{\frac{\log\frac{1-2\delta}{2\delta}}{8T}}, \frac{1}{2}}~~\mbox{and}~~\Pv \defeq \bernoulli\prn*{\frac{1+(2v-1)\eps}{2}}.
	\end{equation*}
	Clearly, $f(\cdot;s)$ is $L$-Lipschitz for both $s\in\{0,1\}$ and moreover,
    \begin{flalign*}
        F_{f,\Pv}(x) - \inf_{\xopt\in\R} F_{f,\Pv}(\xopt) \ge 
        \begin{cases}
        \eps L |x+R| & v = 0 \\
        \eps L |x-R| & v = 1.\\
        \end{cases} \numberthis \label{eq:minimax-lb-error-expression}
    \end{flalign*} 
    Consequently,  $(f, \Pv) \in \Iclass[\Lip]^{L,R}$ for $v\in \{0,1\}$.
    
    For any $\alg\in\Aall$, let $x_T^v$ be the output of $\alg$ when interacting with $(f,\Pv)$ by observing $S_1,\ldots,S_T\simiid \Pv$, and let $\hat{v}(x) \defeq \indic{x>0}$. Substituting into~\eqref{eq:minimax-lb-error-expression}, we have 
    \[
    \ErrHP{\alg}{(f,\Pv)} \ge \eps LR \indicb{ \Pr(\hat{v}\prn*{x_T^v} \ne v) \ge \delta }
    ~~
    \mbox{for}
    ~~
    v\in \{0,1\}.
    \]
    Therefore, since $ \max_{v\in\{0,1\}}\Pr(\hat{v}\prn*{x_T^v} \ne v) \ge \E_{V\sim \bernoulli(2\delta)}\Pr(\hat{v}\prn*{x_T^V} \ne V)$ we have
    \[
    \ErrHP{\alg}{\Iclass[\Lip]^{L,R}} \ge \eps LR \indicb{ \max_{v\in\{0,1\}}\Pr(\hat{v}\prn*{x_T^v} \ne v) \ge \delta } 
    \ge 
    \eps LR \indicb{ \E_{V\sim \bernoulli(2\delta)}\Pr(\hat{v}\prn*{x_T^V} \ne V) \ge \delta }.
    \]
    Let $\hat{V} \defeq \hat{v}\prn*{x_T^V}$ and note that it depends on $V$ only through $x_T^V$ and hence through $S_1,\ldots,S_T\simiid P_V$. By our choice of $\eps$ and the current assumption that $2 \delta < 2/5 \le 1/2$, we may apply \Cref{lem:skewed-binary-bound} with $p=2\delta$ and conclude that $\Pr(\hat{V} \ne V) \ge \delta$, implying
    \[
        \ErrHP{\alg}{\Iclass[\Lip]^{L,R}} \ge \eps LR = \Omega\prn*{  
            LR \min\crl*{ \sqrt{\frac{\log\frac{1}{\delta}}{T}}, 1 }
            }
    \]
    for all $\delta \le \frac{1}{5}$. 

    To extend the lower bound to $\delta \in (\frac{1}{5}, \frac{1}{2})$,
    we only need to show that $\ErrCustom{1/2}{\alg}{\Iclass[\Lip]^{L,R}} = \Omega(LR/\sqrt{T})$. We do this with a simple ``anti-boosting'' argument. Let $k=4$ and consider $k$ copies of the construction above for $\delta'=\frac{1}{5}$, each operating on independent samples and independent coordinates, scaled so that the sum of the copies is in $\Iclass[\Lip]^{L,R}$. Concretely, the construction uses $\xset=\R^k$ and $\sset=\{0,1\}^k$, and given $f:\R\times \{0,1\}\to \R$ described above, defines $\tilde{f}: \xset\times\sset\to\R$ such that
    \[
        \tilde{f}(x,s) = k^{-3/2}\sum_{i \in [k]} f(x_i \sqrt{k};s_i).
    \]
    For $v=(v_1,\ldots,v_k)\in\{0,1\}^k$ and $P=P_{v_1}\times \cdots \times P_{v_k}$, this construction has optimality gap
    \[
        F_{f,\Pv}(x) - \inf_{\xopt\in\xset} F_{f,\Pv}(\xopt)
        \ge \tfrac{1}{k^{3/2}}\eps LR \cdot  \norm*{x \sqrt{k} - (2v-1)R}_1.
    \]
    Therefore, disagreeing with the sign of $\xopt$ in even a single coordinate implies optimality gap of at least $k^{-3/2}\eps LR = \Omega(LR/\sqrt{T})$ by our choice of $\eps$. By the discussion above, the probability of mistaking at least one of the coordinates is at least $1-(1-\delta')^k > \half$ for $\delta' = 1/5$ and $k=4$, completing the proof.
\end{proof} 
\section{Defining the price of adaptivity (PoA)}\label{sec:define-PoA}

We now define and enumerate the basic properties of the novel part of our framework: the price of adaptivity (PoA). 
We begin with the notion of a \emph{meta-class}, i.e., a class of classes, which we denote by $\Mclass$. For any $\ell,\rho \ge 1$, we define
\begin{equation*}
	\Mclass[\Lip]^{\ell, \rho} \defeq 
	\crl*{\Iclass[\Lip]^{L,R} \mid L\in[1, \ell]\text{ and }R\in[1,\rho]}
\end{equation*}
to be the meta-class of Lipschitz SCO problems with uncertainty $\ell$ in the Lipschitz constant and uncertainty $\rho$ in the distance to optimality. Taking the lower Lipschitz and distance bound to be 1 does not compromise generality (see \Cref{prop:PoA-props}.\ref{item:PoA-props:scaling} below). We similarly define $\Mclass[\SMLip]^{\ell, \rho}$ by replacing $\Iclass[\Lip]^{L,R}$ with $\Iclass[\SMLip]^{L,R}$ in the above display.

The PoA of algorithm $\alg$ competing against all SO algorithms on meta-class $\Mclass$ is
\begin{equation*}
	\PoA{\alg}{\Mclass} \defeq \sup_{\Iclass \in \Mclass} 
	\frac{\Err{\alg}{\Iclass}}{\Err{\Aall}{\Iclass}}.
\end{equation*}
We use $\PoAE{\alg}{\Mclass}$ and $\PoAHP{\alg}{\Mclass}$ to denote the price of adaptivity w.r.t.\ expected and quantile error, respectively (i.e., with $\ErrBlank$ standing for $\ErrBlank^{\E}$ and $\ErrBlank^{\delta}$, respectively). 
We may also consider the price of adaptivity in competing against algorithms in $\Aclass\subset\Aall$, which we denote 
\begin{equation}
	\PoAGen{\alg}{\Mclass}{\Aclass} \defeq \sup_{\Iclass \in \Mclass} 
	\frac{\Err{\alg}{\Iclass}}{\Err{\Aclass}{\Iclass}}.
	\label{eq:PoAGen-def}
\end{equation}
Our strategy to lower bound the PoA in this paper is to carefully construct a collection of $n$ `hard' problem instances $(f_1, P_1), \dots, (f_n, P_n)$ such that $(f_k, P_k) \in {\Iclass}_k$ with ${\Iclass}_k \in \Mclass$  and then utilize the following observation
\begin{flalign}\label{eq:PoA-lower-bounding-strategy}
\PoAGen{\alg}{\Mclass}{\Aclass} \ge \max_{k \in [n]} 
\frac{\Err{\alg}{{\Iclass}_k}}{\Err{\Aclass}{{\Iclass}_k}} \ge \max_{k \in [n]} 
\frac{\Err{\alg}{(f_k, P_k)}}{\Err{\Aclass}{{\Iclass}_k}}.
\end{flalign}

\begin{remark}[Information available to algorithms]
	Our setup implicitly gives algorithms complete information about the meta-class. For example, for $\Mclass[\Lip]^{\ell,\rho}$ the algorithm has (potentially very loose) upper and lower bounds on both the Lipschitz constant $L$ and distance to optimality $R$. Existing parameter-free algorithms typically only require a lower bound on $R$ and an upper bound on $L$. Providing algorithms with additional information only strengthens our PoA lower bounds.
\end{remark}
\noindent
We conclude this section with basic properties of the PoA; see \Cref{app:PoA-props-proof} for proof.

\begin{restatable}{prop}{restatePropPoAProps}\label{prop:PoA-props}
The price of adaptivity satisfies the following properties:
\begin{enumerate}[leftmargin=*]
	\item \label{item:PoA-props:trivial-lb} $\PoAGen{\alg}{\Mclass}{\Aclass} \ge 1$ for all $\Mclass,\Aclass$ and $\alg \in\Aclass$.
	\item \label{item:PoA-props:trivial-ub} $\PoAE{\alg_0}{\Mclass} \le O(\sqrt{T})$ and $\PoAHP{\alg_0}{\Mclass} \le O\prn*{\sqrt{\frac{T}{\log\frac{1}{\delta}}+1}}$ for $\Mclass\in\{\Mclass[\Lip]^{\ell,\rho}, \Mclass[\SMLip]^{\ell,\rho}\}$ and $\alg_0$ that trivially returns $x_0=0$. 
	\item \label{item:PoA-props:expectation-higher-PoA} $\PoAE{\alg}{\Mclass}  \ge \Omega(\delta)\PoAHP{\alg}{\Mclass}$ for any $\Mclass\in\{\Mclass[\Lip]^{\ell,\rho}, \Mclass[\SMLip]^{\ell,\rho}\}$, and $\alg\in\Aall$.
	\item \label{item:PoA-props:MS-higher-PoA} $\bigPoA{\alg}{\Mclass[\Lip]^{\ell,\rho}} \le O(1)\bigPoA{\alg}{\Mclass[\SMLip]^{\ell,\rho}}$ for every $\alg\in\Aall$. 
	\item \label{item:PoA-props:All-Afo-same-PoA} $\PoA{\alg}{\Mclass} =
	 \Theta(1)\PoAGen{\alg}{\Mclass}{\AFO}$ for every $\alg\in\Aall$ and $\Mclass\in\{\Mclass[\Lip]^{\ell,\rho}, \Mclass[\SMLip]^{\ell,\rho}\}$.
	\item \label{item:PoA-props:scaling} Let $\Mclass[\Lip]^{\Lmin,\Lmax,\Rmin,\Rmax} \defeq 
	\crl*{\Iclass[\Lip]^{L,R} \mid L\in[\Lmin, \Lmax]\text{ and }R\in[\Rmin,\Rmax]}$.
	Then, for all $\alg\in\Aall$, there exists $\alg'$ s.t.\ 
	$\bigPoA{\alg}{\Mclass[\Lip]^{\Lmin,\Lmax,\Rmin,\Rmax}} = 
	\bigPoA{\alg'}{\Mclass[\Lip]^{\nicefrac{\Lmax}{\Lmin},\nicefrac{\Rmax}{\Rmin}}}$. The same holds replacing $\textup{\textsf{\Lip}}$ with $\textup{\textsf{\SMLip}}$.
\end{enumerate}
\end{restatable}

\section{Proof of lower bounds on the price of adaptivity}\label{sec:lbs}

This section proves our lower bounds; see \Cref{sec:overview-proof-techniques} for proof sketches and \Cref{sec:discussion} for a discussion of the results.

\subsection{Logarithmic PoA for expected error and unknown distance}

\begin{theorem}\label{thm:log-PoA-lb}
For all $T\in\N$, $\rho \ge 1$, $\alg\in\Aall$ we have 
$\bigPoAE{\alg}{\Mclass[\Lip]^{1,\rho}} \ge \Omega\prn[\big]{\min\crl[\big]{\sqrt{\log (\rho)}, \sqrt{T}}}$.
\end{theorem}

\begin{proof}
	We pick two elements from $\Mclass[\Lip]^{1,\rho}$, namely $\IclassMinus = \Iclass[\Lip]^{1, 1}$ and $\IclassPlus = \Iclass[\Lip]^{1, \rho}$. We then pick two instances $(f,\Pminus)$ and $(f,\Pplus)$ such that $(f, \Pv) \in \IclassV$ for $v\in \{0,1\}$. From \cref{eq:PoA-lower-bounding-strategy}, the PoA is bounded from below by
	\begin{equation*}
		\PoAE{\alg}{\Mclass[\Lip]^{1,\rho}} 
		\ge
		\max_{v\in\{0,1\}} \frac{\ErrE{\alg}{(f,\Pv)}}{\ErrE{\Aall}{\IclassV}} \ge \half \cdot \frac{\ErrE{\alg}{(f,\Pminus)}}{\ErrE{\Aall}{\IclassMinus}}
		+  \half \cdot \frac{\ErrE{\alg}{(f,\Pplus)}}{\ErrE{\Aall}{\IclassPlus}}
		.
	\end{equation*}
	Substituting $\ErrE{\Aall}{\Iclass[\Lip]^{L, R}} = O(L R / \sqrt{T} )$ as per Lemma~\ref{prop:standard-error-bounds}, we obtain
	\begin{equation}\label{eq:poa-as-expectation}
		\PoAE{\alg}{\Mclass[\Lip]^{1,\rho}} \ge \frac{\sqrt{T}}{2} \ErrE{\alg}{(f,\Pminus)} + \frac{\sqrt{T}}{2\rho} \ErrE{\alg}{(f,\Pplus)}.
	\end{equation}
	It remains to specify $f$, $\Pminus$ and $\Pplus$, argue that $(f, \Pv) \in \IclassV$ for $v\in \{0,1\}$ holds, and lower bound $\ErrE{\alg}{(f,\Pv)}$. 
	We let $\xset=\R$, $\sset = \{0,1\}$ and $f:\xset\times\sset\to\R$ be
	\begin{equation*}
		f(x;s) = \begin{cases}
			|x| & s = 0 \\
			|x-\rho| & s = 1.\\
		\end{cases}
	\end{equation*}
	Furthermore, we let 
	\begin{equation*}
		\eps \defeq \min\crl*{\sqrt{\frac{\log\rho}{8T}}, \frac{1}{2}}~~\mbox{and}~~\Pv \defeq \bernoulli\prn*{\frac{1+(2v-1)\eps}{2}}.
	\end{equation*}
	Clearly, $f(\cdot;s)$ is 1-Lipschitz for both $s\in\{0,1\}$ and moreover,
\begin{flalign}
	F_{f,\Pv}(x) - \inf_{\xopt\in\R} F_{f,\Pv}(\xopt) \ge 
	\begin{cases}
	\eps |x| & v = 0 \\
	\eps |x-\rho| & v = 1.\\
	\end{cases} \label{eq:expectation-lb-error-expression}
	\end{flalign} 
	Consequently,  $(f, \Pv) \in \IclassV$ for $v\in \{0,1\}$ as required. 

	Let $S_1, \dots, S_T\sim P_v$, yielding the  sample functions $f(\cdot, S_1), \dots, f(\cdot, S_T)$.
	To lower bound the error, let $x_T$ be the output of the algorithm after $T$ oracle queries. Using the expression~\eqref{eq:expectation-lb-error-expression}, we have
	\begin{equation*}
		\ErrE{\alg}{(f,\Pminus)} \ge \frac{\eps\rho}{2} \cdot 
		\Pminus\prn*{x_T \ge \frac{\rho}{2}}
	\quad \text{and} \quad 
		\ErrE{\alg}{(f,\Pplus)} \ge \frac{\eps\rho}{2} \cdot 
		\Pplus\prn*{x_T < \frac{\rho}{2}}.
	\end{equation*}
	To translate these lower bounds into the language of hypothesis testing, let $\hat{V} \defeq \indicb{x_T \ge \frac{\rho}{2}}$
	and note that, for $v\in\{0,1\}$,
	\[
		\ErrE{\alg}{(f,\Pv)} \ge \frac{\eps\rho}{2} \Pv(\hat{V} \ne v ).
	\]
	Substituting into~\eqref{eq:poa-as-expectation}, we have that
		$\bigPoAE{\alg}{\Mclass[\Lip]^{1,\rho}} \ge \frac{\eps\sqrt{T}}{4}\prn[\big]{
			\rho \Pminus(\hat{V} \ne 0) + \Pplus(\hat{V} \ne 1)
		}$.
	Writing $p \defeq \frac{1}{1+\rho}$ and letting $V\sim \bernoulli(p)$, we may rewrite the above PoA as
	\[
	\PoAE{\alg}{\Mclass[\Lip]^{1,\rho}} \ge \frac{\eps\sqrt{T}}{4p} 
		\prn[\Big]{
			(1-p) \Pminus(\hat{V} \ne 0) + p \Pplus(\hat{V} \ne 1)
		}
		= \frac{\eps\sqrt{T}}{4p} \P(\hat{V} \ne V).
	\]
	Noting that $p\le \half$ since $\rho \ge 1$ and that $\rho = \frac{1-p}{p}$, we invoke \Cref{lem:skewed-binary-bound} (in \Cref{app:skewed-binary-bound}) to conclude that $\P(\hat{V} \ne V) \ge \frac{p}{2}$ and therefore (recalling our choice of $\veps$)
	$\PoAE{\alg}{\Mclass[\Lip]^{1,\rho}} \ge \frac{\eps \sqrt{T}}{8} \ge 
		\frac{1}{32}\min\crl{
		\sqrt{\log\rho}, \sqrt{T}
		}$
	as required.
\end{proof}

\subsection{Double logarithmic PoA when Lipschitz constant is known} 
\begin{theorem}\label{thm:loglog-PoA-lb}
	For all $T\in\N$, $\rho \ge 1$, $\delta\in(0,\half]$, and $\alg\in\AFO$ we have 
	\begin{equation*}
		\PoAHP{\alg}{\Mclass[\Lip]^{1,\rho}}
		\ge 
		\PoAHP{\alg}{\Mclass[\Lip]^{1,\rho}; \AFO} \ge 
        \Omega\prn*{
        \sqrt{
            \frac{\min\crl*{{\log \prn*{\frac{1}{\delta}\log \rho} }, {T}}}
        {{\log(\frac{1}{\delta})}}
        }
        }.
	\end{equation*}
\end{theorem}

\begin{proof}
		Throughout the proof we set 
	$n \defeq \ceil{\log\rho}$ and for $k\in[n]$, define $r_k \defeq e^{k-1}$,
	so that $r_1 = 1$ and $r_n \le \rho$. The index $k$ will continue to denote a number in $[n]$, and we will use $K$ for an index $k$ drawn uniformly from $[n]$. 

	To begin, note that it suffices to prove  $\PoAHP{\alg}{\Mclass[\Lip]^{1,\rho}; \AFO} \ge 
	\Omega\prn*{
	\sqrt{
		\frac{\min\crl*{{\log \prn*{\log \rho} }, {T}}}
	{{\log({1}/{\delta})}}
	}
	},$
	since by \Cref{prop:PoA-props}.\ref{item:PoA-props:trivial-lb} we know that $\PoAHP{\alg}{\Mclass[\Lip]^{1,\rho}; \AFO} \ge 1$ and $\max\crl*{1, \frac{\log \log \rho}{\log (1/\delta)}} \ge \half  
	\frac{\log \prn*{ \frac{1}{\delta}\log \rho}}{\log (1/\delta)}$.

	We pick $n$ elements from $\Mclass[\Lip]^{1,\rho}$, namely $\IclassK = \Iclass[\Lip]^{1, r_k}$. For each element, we construct an instance $(\fk, P)$ such that $(\fk, P) \in \IclassK$. Since $\IclassK \in \Mclass[\Lip]^{1,\rho}$ for all $k\in[n]$, from \cref{eq:PoA-lower-bounding-strategy} we have
		$\PoAHP{\alg}{\Mclass[\Lip]^{1,\rho}; \AFO} \ge 
		\max_{k\in[n]} \frac{\ErrHP{\alg}{(\fk,P)}}{\ErrHP{\AFO}{\IclassK}}$.  
	Substituting $\ErrHP{\AFO}{\IclassK} = \O{ r_k \sqrt{\frac{\log\prn*{({1}/{\delta}}}{T}}}$ from \Cref{prop:standard-error-bounds}, we obtain
	\begin{flalign}
		\PoAHP{\alg}{\Mclass[\Lip]^{1,\rho} ; \AFO} & \ge \Omega\prn*{
		\sqrt{\frac{T}{\log ({1}/{\delta})}}
        }
        \max_{k\in[n]} \frac{\ErrHP{\alg}{(\fk,P)}}{r_k} .
		\label{eq:loglog-PoA-arguement}
	\end{flalign}

	Let us now construct $\{\fk\}$ and $P$. First, define
	\begin{equation*}
		\veps = \min\crl*{ \sqrt{\frac{\log n}{4T}}, 1 }.
	\end{equation*} 
	We let the domain $\xset=\R$, the sample space $\sset=\{-1,0,1\}$, and set:
	\begin{equation*}
	\fk(x,s) \defeq \begin{cases}
		-x & s = -1 \\
		\abs*{x-r_k} & s = 0 \\
		x & s = 1, \\
	\end{cases}
	~~\mbox{and}~~P(\pm 1) = \frac{1-\veps}{2},~P(0)= \veps.
	\end{equation*} 
	Clearly, $\fk(\cdot, s)$ is 1-Lipschitz for all $s$ and $k$. Moreover, we have
	\begin{equation*}
		F_{\fk,P}(x) - \inf_{\xopt\in\R} F_{\fk,P}(\xopt) = \veps \abs*{x-r_k}
		\numberthis
		\label{eq:loglog-construction-subopt}
	\end{equation*}
	and therefore $(\fk,P)\in\IclassK$ as required. 

	For any $x\in\R$, let 
	\[
	\hat{K}(x) \defeq \argmin_{k\in[n]}\abs{x-r_k}
	\]
	denote the index $k$ for which $r_k$ is nearest to $x$. Letting $r_0=-\infty$ and $r_{n+1}=\infty$, if $\hat{K}(x) \ne k$, we know that
	\begin{equation*}
		\abs{x-r_k} \ge \min\crl*{ \frac{r_k - r_{k-1}}{2}, \frac{r_{k+1} - r_k }{2}} \ge r_k\min\crl*{ \frac{1 - e^{-1}}{2}, \frac{e - 1 }{2}} \ge \frac{r_k}{4}. 
	\end{equation*}
	Letting $x_T^k$ denote the output of $\alg$ after interacting with $(f_k, P)$ for $T$ iterations and recalling~\eqref{eq:loglog-construction-subopt}, we observe that the event $\hat{K}(x_T^k)\ne k$ implies $F_{\fk,P}(x_T^k) - \inf_{\xopt\in\R} F_{\fk,P}(\xopt) \ge \veps r_k /4$. In other words, 
		$\ErrHP{\alg}{(\fk,P)} \ge \frac{\veps r_k}{4} \indic{\Pr(\hat{K}(x_T^k)\ne k) \ge \delta}$.
	Substituting back into~\eqref{eq:loglog-PoA-arguement} and recalling the definition of $\veps$, we have
	\begin{equation*}
		\PoAHP{\alg}{\Mclass[\Lip]^{1,\rho} ; \AFO} \ge \Omega\prn*{
			\eps\sqrt{\frac{T}{\log ({1}/{\delta})}}\max_{k\in[n]} \indicb{\Pr(\hat{K}(x_T^k)\ne k) \ge \delta}}.
	\end{equation*}
	Let $K\sim \uniform([n])$ and, to streamline notation, let $\hat{K} = \hat{K}(x_T^K)$. Noting that
		$\Pr(\hat{K} \ne K) = \frac{1}{n}\sum_{k=1}^n \Pr(\hat{K}(x_T^k)\ne k)$
	we observe that if $\Pr(\hat{K} \ne K) \ge \delta$ then also $\Pr(\hat{K}(x_T^k)\ne k) \ge \delta$ for some $k \in [n]$, and hence, by definition of $n$ and $\eps$,
	\begin{equation*}
		\PoAHP{\alg}{\Mclass[\Lip]^{1,\rho}; \AFO} \ge \Omega\prn*{
			\eps\sqrt{\frac{T}{\log(1/\delta)}} \indicb{\Pr(\hat{K}\ne K) \ge \delta} }
			= \Omega\prn*{ \sqrt{
                \frac{\min\crl*{{\log \log \rho }, {T}}}
            {{\log({1}/{\delta})}}
            }\indicb{\Pr(\hat{K}\ne K) \ge \delta} }.
	\end{equation*}

	Finally, we show that $\Pr(\hat{K} \ne K) \ge \delta$ holds for all $\alg\in\AFO$. To that end, let $S_1, \ldots, S_{T}\simiid P$ be the sequence of observed samples generating, and note that the algorithm's output $x_T^K$ is a randomized function of the stochastic (sub)gradients $g_0, \ldots, g_{T-1}$ where $g_i \in \del f_K(x_{i}; S_{i+1})$.  We may choose the subgradients of $\fK$ such that they only take values in $\{-1,1\}$ and, for all $t\ge 0$, if we write
		$S_{t+1}' \defeq  \frac{g_t +1 }{2}$
	then the distribution of $S_t'$ is 
		$\bernoulli(\frac{1-\veps}{2})$ if $x_{t-1} < r_K$, and $\bernoulli(\frac{1+\veps}{2})$ otherwise. %
	Importantly, when conditioning on $S_1', \ldots, S_{t-1}'$ and $K$ (and even if we further condition on $x_{t-1}$), the probability that $S_t'=0$ is always at most $\frac{\eps}{2}$ away from $\frac{1}{2}$. Therefore, 
	viewing $x_T^K$ and hence $\hat{K}$ as a (randomized) function of $S_1', \ldots, S_{T}'$, and assuming $n>16$ without loss of generality, \Cref{lem:noisy-binary-search-bound} (in \Cref{app:noisy-binary-search-bound}) yields $\Pr(\hat{K}\ne K) > \half \ge \delta$ as required.
\end{proof}

\subsection{Polynomial PoA when distance and Lipschitz constant are both unknown}

\begin{theorem}\label{thm:poly-PoA-lb}
	For all $T\in\N$, $\ell\ge 1$, $\rho \ge 1$, $\delta\in(0,\frac{1}{3})$ and $\alg\in\Aall$ we have
	\begin{equation*}
		\PoAHP{\alg}{\Mclass[\Lip]^{\ell,\rho}} \ge \Omega\prn*{ 
			\min\crl*{ \frac{\min\{\ell, \rho\}\sqrt{\log\frac{1}{\delta}}}{\sqrt{T}}, \sqrt{\frac{T}{\log\frac{1}{\delta}}}}}
	\end{equation*}
and
\begin{equation*}
	\PoAHP{\alg}{\Mclass[\SMLip]^{\ell,\rho}}\ge \Omega\prn*{ 
		\min\crl*{\rho, \ell, \sqrt{\frac{T}{\log\frac{1}{\delta}}} }}	.
\end{equation*}
\end{theorem}

\newcommand{\lfp}{\log_4(1/(2\delta))}
\newcommand{\lf}{\log_4\frac{1}{2\delta}}
\newcommand{\ellmin}{\ell \wedge \rho}
\newcommand{\ToverLog}{\frac{T}{\log\frac{1}{\delta}}}

\begin{proof}
	We begin by describing two instances $(f, P_0)$ and $(f, P_1)$ that are hard to distinguish. To that end, define $\lambda \defeq \min\crl*{\lf, \frac{T}{2}}$ and for some $\alpha \le 1$ to be determined, let 
\begin{equation*}
	f(x;s) = \begin{cases}
		\alpha \abs{x-\rho} & s = 0\\
		\frac{2\alpha T}{\lambda} \abs{x} & s = 1.
	\end{cases}
\end{equation*}
For $v \in \{0, 1\}$ define $P_v = \bernoulli\prn*{\frac{v\lambda}{T}}$.
That is, $\Pzero(0)=1$ and $\Pone(0) = 1-\frac{\lambda}{T}$. 
For all $x \in \R$,
\begin{flalign*}
F_{f,\Pzero}(x) - \inf_{x_\star \in \R} F_{f,\Pzero}(x_\star) &=  \alpha  \abs{x - \rho}~~\mbox{and,}\\
F_{f,\Pone}(x) - \inf_{x_\star \in \R} F_{f,\Pone}(x_\star) &=  \left( 1 - \frac{\lambda}{T} \right) \alpha \prn*{\abs{x - \rho} - \rho} + 2 \alpha \abs{x} \ge \alpha \abs{x}.
\end{flalign*}

Using the above suboptimality expression, we bound the error incurred by any algorithm. %
The algorithm's output $x_T$ satisfies 
\begin{equation}\label{eq:thm:PoA-dist-Lip-both-unknown:initial-lb-errors}
	\ErrHP{\alg}{(f,\Pzero)} \ge \frac{\alpha\rho}{2} \indicb{P_0\prn*{ x_T \le \frac{\rho}{2}} \ge \delta}
	~~\mbox{and}~~
	\ErrHP{\alg}{(f,\Pone)} \ge \frac{\alpha\rho}{2} \indicb{P_1\prn*{ x_T > \frac{\rho}{2}} \ge \delta}.
\end{equation}
Writing 
\[
q \defeq \Pr*(x_T \le \frac{\rho}{2} | S_1 = \cdots = S_T = 0 ),
\]
we note that, since $P_0(S_1 = \cdots = S_T = 0 )=1$, we have
$P_0\prn*{x_T \le \frac{\rho}{2}} = q.$
Moreover, using the choice of $\lambda$ and  $1-x\ge 4^{-x}$ for all $x\le \half$, we have 
$P_1(S_1 = \cdots = S_T = 0 ) = \prn*{1-\frac{\lambda}{T}}^T \ge 4^{-\lambda} \ge 2\delta$. Therefore, 
\[
\Pone\prn*{x_T > \frac{\rho}{2}} \ge \Pr*(x_T > \frac{\rho}{2} | S_1 = \cdots = S_T = 0) P_1(S_1 = \cdots = S_T = 0 ) \ge 2\delta(1-q).
\]
Substituting $P_0\prn*{x_T \le \frac{\rho}{2}} = q$ and $\Pone\prn*{x_T > \frac{\rho}{2}} \ge 2\delta(1-q)$ into \cref{eq:thm:PoA-dist-Lip-both-unknown:initial-lb-errors} gives
\begin{equation}\label{eq:thm:PoA-dist-Lip-both-unknown:final-lb-errors}
	\ErrHP{\alg}{(f,\Pzero)} \ge \frac{\alpha\rho}{2} \indicb{q \ge \delta}
	~~\mbox{and}~~
	\ErrHP{\alg}{(f,\Pone)} \ge \frac{\alpha\rho}{2} \indicb{2\delta (1-q) \ge \delta} = \frac{\alpha\rho}{2} \indicb{q \le \frac{1}{2}}.
\end{equation}

Next, we associate our constructed instances with instance classes and bound the PoA from below. 
We note that $F_{f,P_v}$ is minimized at $\xopt = \rho$ for $v=0$ and $\xopt=0$ for $v=1$, and that $f(\cdot,0)$ and $f(\cdot,1)$ are $\alpha$ and $2\alpha T / \lambda$ Lipschitz, respectively. Therefore, by setting
\begin{equation*}
	\alpha =  \min\crl*{1, \frac{\lambda}{2T}\ellmin}
	~~\mbox{we get}~~
(f,P_0)\in \Iclass[\Lip]^{1, \rho}
~~\mbox{and}~~
(f,P_1)\in \Iclass[\Lip]^{\ellmin, 1}, %
\end{equation*}
where 
$\ellmin \defeq \min\{\ell,\rho\}$.
Since $\Iclass[\Lip]^{1, \rho}$ and $\Iclass[\Lip]^{\ellmin, 1}$ are both members of $\Mclass[\Lip]^{\ell,\rho}$, we have by \cref{eq:PoA-lower-bounding-strategy}, \cref{eq:thm:PoA-dist-Lip-both-unknown:final-lb-errors} and $\ErrHP{\Aall}{\Iclass[\Lip]^{L, R}} = \O{\frac{L R}{\sqrt{T}}\sqrt{\log\frac{1}{\delta}}}$ (see \Cref{prop:standard-error-bounds}) that
\begin{flalign*}
	\PoAHP{\alg}{\Mclass[\Lip]^{\ell,\rho}} 
	& \ge \max\crl*{\frac{\ErrHP{\alg}{(f,\Pzero)}}{\ErrHP{\Aall}{\Iclass[\Lip]^{1, \rho}}}, \frac{\ErrHP{\alg}{(f,\Pone)}}{\ErrHP{\Aall}{\Iclass[\Lip]^{\ellmin, 1}}}} 
	\\
	& \ge 
	\Omega(1) 
		 \max\crl*{{\alpha\sqrt{\ToverLog}}\indic{q\ge \delta}, 
	\frac{\alpha \rho}{\ellmin}\sqrt{\ToverLog}\indic{q\le \frac{1}{2}}
	}
    =
	\Omega\prn*{\alpha \sqrt{\ToverLog}}
	, \numberthis \label{eq:poly-PoA-argument}
\end{flalign*}
with the last transition due to $\frac{\rho}{\ellmin} \ge 1$ and the fact that either $q\ge \delta$ or $q\le \frac{1}{2}$ holds for all $q$. Substituting our choices of $\alpha$ and $\lambda$ yields the claimed lower bound for the Lipschitz case.

Moving on to $\Mclass[\SMLip]^{\ell,\rho}$, we observe that $\E_{S\sim P_0} \norm{\grad f(x; S)}^2 = \alpha^2$
and 
\[
\E_{S\sim P_1} \norm{\grad f(x; S)}^2 = 
\frac{\lambda}{T} \cdot 
\prn*{\frac{2\alpha T}{\lambda}}^2 + \prn*{1-\frac{\lambda}{T}} \alpha^2 \le \frac{5\alpha^2 T}{\lambda}.
\]
Therefore, setting
$\alpha = \min\crl[\Big]{1, \sqrt{\frac{\lambda}{5T}}(\ellmin)}$
we get $(f,P_0)\in \Iclass[\SMLip]^{1, \rho}$
and 
$(f,P_1)\in \Iclass[\SMLip]^{\ellmin, 1}$.
According to \Cref{prop:standard-error-bounds}, the bounded second-moment assumption leads to the same known-parameter minimax rates as the bounded Lipschitz assumption, i.e., $\ErrHP{\Aall}{\Iclass[\SMLip]^{L, R}} = \O{\frac{L R}{\sqrt{T}}\sqrt{\log\frac{1}{\delta}}}$. Therefore, since  $\Iclass[\SMLip]^{1, \rho}, \Iclass[\SMLip]^{\ellmin, 1}\in \Mclass[\SMLip]^{\ell,\rho}$, we may repeat the argument in \cref{eq:poly-PoA-argument} and conclude that 
		$\PoAHP{\alg}{\Mclass[\Lip]^{\ell,\rho}} = \Omega\prn*{\alpha \sqrt{T/\log\tfrac{1}{\delta}}}=\Omega\prn*{\min\crl*{\sqrt{T/\log\tfrac{1}{\delta}}, \ellmin}}$,
establishing the claimed lower bound in the bounded second moment setting.
\end{proof}

\section{Discussion}\label{sec:discussion}

Below, we describe the main conclusions from our work and open problems they put into focus.

\paragraph{The PoA in stochastic convex optimization is nearly settled.}
In the regime where the Lipschitz constant (i.e., probability $1$ stochastic gradient bound) is known ($\ell=O(1)$), \Cref{thm:log-PoA-lb,thm:loglog-PoA-lb} show that \cite{mcmahan2014unconstrained} and \cite{carmon2022making} provide optimal and near-optimal adaptivity to uncertainty in the distance to optimality ($\rho$), for expected and constant-probability error, respectively. Moreover, when the distance to optimality is known ($\rho=O(1)$) it is possible to be completely adaptive to the Lipschitz constant~\cite{gupta2017unified}. 
Finally, the polynomial PoA lower bounds in \Cref{thm:poly-PoA-lb} are matched, up to polylogarithmic factors, by a combination of existing algorithms~\cite{mcmahan2014unconstrained,carmon2022making,cutkosky2019artificial}. Thus, our work establishes that overheads incurred by known adaptive and parameter-free algorithms are for the most part unavoidable. In particular, when both $\ell$ and $\rho$ are $\Omega(T)$, nontrivial adaptivity is impossible.

\paragraph{Adaptivity is much harder with heavy-tailed noise.}
\Cref{thm:poly-PoA-lb} also shows that the price of adaptivity becomes much higher when we replace the assumption that all stochastic gradients are bounded with probability 1 with the assumption that only their second moment is bounded, allowing for heavy-tailed noise. Under the latter assumption, we show that any algorithm must have PoA at least linear in---and hence be sensitive to---the uncertainty in either gradient bound or distance to optimality, while under the former assumption, the PoA is smaller by a factor of $\sqrt{T}$, allowing for robustness to both parameters in some regimes. 
Our separation between Lipschitz assumptions is notable because when the problem parameters are known the rates of convergence under both assumptions are essentially the same, both in expectation and with high probability (\Cref{prop:standard-error-bounds}). Thus, robustness to heavy-tailed noise must come at the cost of knowing some problem parameters.

\paragraph{PoA in high-probability is lower (!) than in expectation.}
\Cref{thm:log-PoA-lb} implies that, for uncertainty $\rho$ in distance to optimality, any adaptive algorithm must have expected suboptimality larger by factor of $\Omega(\sqrt{\log \rho})$ than the minimax optimal rate with known parameters. This may appear to contradict~\citet{carmon2022making}, who give probability $1-\delta$ suboptimality bounds only a factor of $O(\log^2 (\log (\rho T) /\delta))$ larger than the optimal rate. The apparent contradiction stems from the fact that high-probability bounds are typically stronger than in-expectation bounds. Indeed, for any random variable $Z$ (e.g., the suboptimality of an algorithm's output) if the $1-\delta$ quantile $Q_{1-\delta}(Z)$ is bounded by $C \mathrm{poly}\prn*{{\log} \frac{1}{\delta}}$ \emph{for all} $\delta$, then the identity $\E Z = \int_{0}^1 Q_\delta(Z)\mathrm{d} \delta$ implies that $\E Z = O(C)$. 
The resolution to this apparent contradiction is that the bound in~\cite{carmon2022making} does not hold simultaneously for all $\delta$; instead, the desired confidence level $\delta$ is prespecified and affects the algorithm. Thus, the best possible probability $1-\delta$ PoA bound holding \emph{uniformly} for all $\delta$ must be logarithmic in $\rho$; characterizing the uniform-high-probability PoA is an open problem.

\paragraph{Sample complexity vs.\ gradient oracle complexity.} \Cref{thm:log-PoA-lb,thm:poly-PoA-lb} are \emph{sample complexity} lower bounds: they are valid for algorithms with unrestricted access to the sample functions. In contrast, \Cref{thm:loglog-PoA-lb} is an \emph{oracle complexity} lower bound: it holds for algorithms restricted to evaluating a single gradient for each sample function. It is straightforward to extend the construction proving \Cref{thm:loglog-PoA-lb} to algorithms that observe a sample function value in addition to a gradient. To do so, we may simply add a random constant offset to each sample function and bound the information that observing the function value can add. The required random offset might be fairly large, but the natural assumptions for our problem place no limitation on its size. Fortifying the construction of \Cref{thm:loglog-PoA-lb} against algorithms with complete sample function access appears to be much more difficult, since for such algorithms the first draw of $s=0$ reveals the optimum. Whether the $\log \log \rho$ PoA lower bound also holds for sample complexity remains an open problem.

\paragraph{The price of adaptivity beyond non-smooth stochastic convex optimization.} 
This work considers meta-classes that capture the well-studied setting of non-smooth stochastic convex optimization. However, additional assumptions such as smoothness, strong convexity, and noise variance impact convergence guarantees and introduce additional problem parameters that are rarely known in advance. Consequently, to fully understand the PoA in stochastic convex optimization, we must explore richer meta-classes. Furthermore, it is possible to study the PoA beyond convex optimization, e.g., by replacing the function-value gap with gradient norm.

\newcommand{\theacks}{
	We thank Amit Attia, Chi Jin, Ahmed Khaled, and Tomer Koren for helpful discussions.
	This work was supported by the NSF-BSF program, under NSF grant \#2239527
	and BSF grant \#2022663.
	OH acknowledges support from Pitt Momentum Funds,  and
	AFOSR grant \#FA9550-23-1-0242.
	YC acknowledges support from the Israeli Science
	Foundation (ISF) grant no. 2486/21, the Alon Fellowship, and the Adelis Foundation.
}

\notarxiv{
\acks{\theacks}
}
\arxiv{
\newpage
\section*{Acknowledgements}
\theacks
}

\notarxiv{\newpage}
\arxiv{\bibliographystyle{abbrvnat}}

\begin{thebibliography}{47}
\providecommand{\natexlab}[1]{#1}
\providecommand{\url}[1]{\texttt{#1}}
\expandafter\ifx\csname urlstyle\endcsname\relax
  \providecommand{\doi}[1]{doi: #1}\else
  \providecommand{\doi}{doi: \begingroup \urlstyle{rm}\Url}\fi

\bibitem[Agarwal et~al.(2012)Agarwal, Bartlett, Ravikumar, and
  Wainwright]{agarwal2012information}
A.~Agarwal, P.~L. Bartlett, P.~Ravikumar, and M.~J. Wainwright.
\newblock Information-theoretic lower bounds on the oracle complexity of
  stochastic convex optimization.
\newblock \emph{IEEE Transactions on Information Theory}, 58\penalty0
  (5):\penalty0 3235--3249, 2012.

\bibitem[Attia and Koren(2024)]{attia2024free}
A.~Attia and T.~Koren.
\newblock How free is parameter-free stochastic optimization?
\newblock \emph{arXiv:2402.03126}, 2024.

\bibitem[Beck and Teboulle(2009)]{beck2009fast}
A.~Beck and M.~Teboulle.
\newblock A fast iterative shrinkage-thresholding algorithm for linear inverse
  problems.
\newblock \emph{SIAM journal on imaging sciences}, 2\penalty0 (1):\penalty0
  183--202, 2009.

\bibitem[Borovkov(1999)]{borovkov1999probability}
A.~A. Borovkov.
\newblock \emph{Probability theory}.
\newblock CRC Press, 1999.

\bibitem[Braun et~al.(2017)Braun, Guzm{\'a}n, and Pokutta]{braun2017lower}
G.~Braun, C.~Guzm{\'a}n, and S.~Pokutta.
\newblock Lower bounds on the oracle complexity of nonsmooth convex
  optimization via information theory.
\newblock \emph{IEEE Transactions on Information Theory}, 63\penalty0
  (7):\penalty0 4709--4724, 2017.

\bibitem[Carmon and Hinder(2022)]{carmon2022making}
Y.~Carmon and O.~Hinder.
\newblock Making {SGD} parameter-free.
\newblock In \emph{Conference on Learning Theory (COLT)}, 2022.

\bibitem[Carmon et~al.(2022)Carmon, Hausler, Jambulapati, Jin, and
  Sidford]{carmon2022optimal}
Y.~Carmon, D.~Hausler, A.~Jambulapati, Y.~Jin, and A.~Sidford.
\newblock Optimal and adaptive monteiro-svaiter acceleration.
\newblock In \emph{Advances in Neural Information Processing Systems
  (NeurIPS)}, 2022.

\bibitem[Chen et~al.(2022)Chen, Langford, and Orabona]{chen2022better}
K.~Chen, J.~Langford, and F.~Orabona.
\newblock Better parameter-free stochastic optimization with {ODE} updates for
  coin-betting.
\newblock In \emph{AAAI Conference on Artificial Intelligence}, 2022.

\bibitem[Cover and Thomas(1991)]{cover1991elements}
T.~M. Cover and A.~J. Thomas.
\newblock \emph{Elements of information theory}.
\newblock John Wiley \& Sons, 1991.

\bibitem[Cutkosky(2019)]{cutkosky2019artificial}
A.~Cutkosky.
\newblock Artificial constraints and hints for unbounded online learning.
\newblock In \emph{Conference on Learning Theory (COLT)}, 2019.

\bibitem[Cutkosky and Boahen(2017)]{cutkosky2017online}
A.~Cutkosky and K.~Boahen.
\newblock Online learning without prior information.
\newblock In \emph{Conference on Learning Theory (COLT)}, 2017.

\bibitem[Cutkosky and Orabona(2018)]{cutkosky2018black}
A.~Cutkosky and F.~Orabona.
\newblock Black-box reductions for parameter-free online learning in {B}anach
  spaces.
\newblock In \emph{Conference on Learning Theory (COLT)}, 2018.

\bibitem[Davis et~al.(2021)Davis, Drusvyatskiy, Xiao, and Zhang]{davis2021low}
D.~Davis, D.~Drusvyatskiy, L.~Xiao, and J.~Zhang.
\newblock From low probability to high confidence in stochastic convex
  optimization.
\newblock \emph{The Journal of Machine Learning Research}, 22\penalty0
  (1):\penalty0 2237--2274, 2021.

\bibitem[Duchi(2018)]{duchi2018introductory}
J.~C. Duchi.
\newblock Introductory lectures on stochastic optimization.
\newblock \emph{The Mathematics of Data}, 25:\penalty0 99--186, 2018.

\bibitem[Fano and Wintringham(1961)]{fano1961transmission}
R.~M. Fano and W.~Wintringham.
\newblock Transmission of information, 1961.

\bibitem[Freedman(1975)]{freedman1975tail}
D.~A. Freedman.
\newblock On tail probabilities for martingales.
\newblock \emph{the Annals of Probability}, pages 100--118, 1975.

\bibitem[Gorbunov et~al.(2020)Gorbunov, Danilova, and
  Gasnikov]{gorbunov2020stochastic}
E.~Gorbunov, M.~Danilova, and A.~Gasnikov.
\newblock Stochastic optimization with heavy-tailed noise via accelerated
  gradient clipping.
\newblock In \emph{Advances in Neural Information Processing Systems
  (NeurIPS)}, 2020.

\bibitem[Gorbunov et~al.(2021)Gorbunov, Danilova, Shibaev, Dvurechensky, and
  Gasnikov]{gorbunov2021near}
E.~Gorbunov, M.~Danilova, I.~Shibaev, P.~Dvurechensky, and A.~Gasnikov.
\newblock Near-optimal high probability complexity bounds for non-smooth
  stochastic optimization with heavy-tailed noise.
\newblock \emph{arXiv:2106.05958}, 2021.

\bibitem[Gupta et~al.(2017)Gupta, Koren, and Singer]{gupta2017unified}
V.~Gupta, T.~Koren, and Y.~Singer.
\newblock A unified approach to adaptive regularization in online and
  stochastic optimization.
\newblock \emph{arXiv:1706.06569}, 2017.

\bibitem[Hazan and Kakade(2019)]{hazan2019revisiting}
E.~Hazan and S.~Kakade.
\newblock Revisiting the {P}olyak step size.
\newblock \emph{arXiv:1905.00313}, 2019.

\bibitem[Ivgi et~al.(2023)Ivgi, Hinder, and Carmon]{ivgi2023dog}
M.~Ivgi, O.~Hinder, and Y.~Carmon.
\newblock {D}o{G} is {SGD}'s best friend: A parameter-free dynamic step size
  schedule.
\newblock In \emph{International Conference on Machine Learning (ICML)}, 2023.

\bibitem[Jacobsen and Cutkosky(2023)]{jacobsen2023unconstrained}
A.~Jacobsen and A.~Cutkosky.
\newblock Unconstrained online learning with unbounded losses.
\newblock In \emph{International Conference on Machine Learning (ICML)}, 2023.

\bibitem[Karp and Kleinberg(2007)]{karp2007noisy}
R.~M. Karp and R.~Kleinberg.
\newblock Noisy binary search and its applications.
\newblock In \emph{Symposium on Discrete Algorithms (SODA)}, 2007.

\bibitem[Kavis et~al.(2019)Kavis, Levy, Bach, and Cevher]{kavis2019unixgrad}
A.~Kavis, K.~Y. Levy, F.~Bach, and V.~Cevher.
\newblock Uni{XG}rad: A universal, adaptive algorithm with optimal guarantees
  for constrained optimization.
\newblock \emph{Advances in Neural Information Processing Systems (NeurIPS)},
  2019.

\bibitem[Khaled and Jin(2024)]{khaled2024tuning}
A.~Khaled and C.~Jin.
\newblock Tuning-free stochastic optimization.
\newblock \emph{arXiv:2402.07793}, 2024.

\bibitem[McMahan(2017)]{mcmahan2017survey}
H.~B. McMahan.
\newblock A survey of algorithms and analysis for adaptive online learning.
\newblock \emph{The Journal of Machine Learning Research}, 18\penalty0
  (1):\penalty0 3117--3166, 2017.

\bibitem[McMahan and Orabona(2014)]{mcmahan2014unconstrained}
H.~B. McMahan and F.~Orabona.
\newblock Unconstrained online linear learning in {H}ilbert spaces: Minimax
  algorithms and normal approximations.
\newblock In \emph{Conference on Learning Theory (COLT)}, 2014.

\bibitem[Mhammedi and Koolen(2020)]{mhammedi2020lipschitz}
Z.~Mhammedi and W.~M. Koolen.
\newblock {L}ipschitz and comparator-norm adaptivity in online learning.
\newblock In \emph{Conference on Learning Theory (COLT)}, 2020.

\bibitem[Mhammedi et~al.(2019)Mhammedi, Koolen, and
  Van~Erven]{mhammedi2019lipschitz}
Z.~Mhammedi, W.~M. Koolen, and T.~Van~Erven.
\newblock Lipschitz adaptivity with multiple learning rates in online learning.
\newblock In \emph{Conference on Learning Theory}, pages 2490--2511, 2019.

\bibitem[Mishchenko and Defazio(2023)]{mishchenko2023prodigy}
K.~Mishchenko and A.~Defazio.
\newblock Prodigy: An expeditiously adaptive parameter-free learner.
\newblock \emph{arXiv:2306.06101}, 2023.

\bibitem[Nazin et~al.(2019)Nazin, Nemirovsky, Tsybakov, and
  Juditsky]{nazin2019algorithms}
A.~V. Nazin, A.~S. Nemirovsky, A.~B. Tsybakov, and A.~B. Juditsky.
\newblock Algorithms of robust stochastic optimization based on mirror descent
  method.
\newblock \emph{Automation and Remote Control}, 80:\penalty0 1607--1627, 2019.

\bibitem[Nemirovski and Yudin(1983)]{nemirovski1983problem}
A.~Nemirovski and D.~Yudin.
\newblock \emph{Problem complexity and method efficiency in optimization}.
\newblock Wiley-Interscience, 1983.

\bibitem[Nguyen et~al.(2023)Nguyen, Nguyen, Ene, and
  Nguyen]{nguyen2023improved}
T.~D. Nguyen, T.~H. Nguyen, A.~Ene, and H.~Nguyen.
\newblock Improved convergence in high probability of clipped gradient methods
  with heavy tailed noise.
\newblock In \emph{Advances in Neural Information Processing Systems
  (NeurIPS)}, 2023.

\bibitem[Orabona(2013)]{orabona2013dimension}
F.~Orabona.
\newblock Dimension-free exponentiated gradient.
\newblock \emph{Advances in Neural Information Processing Systems (NeurIPS)},
  2013.

\bibitem[Orabona(2014)]{orabona2014simultaneous}
F.~Orabona.
\newblock Simultaneous model selection and optimization through parameter-free
  stochastic learning.
\newblock \emph{Advances in Neural Information Processing Systems (NeurIPS)},
  2014.

\bibitem[Orabona(2021)]{orabona2021modern}
F.~Orabona.
\newblock A modern introduction to online learning.
\newblock \emph{arXiv:1912.13213}, 2021.

\bibitem[Orabona and P{\'a}l(2016)]{orabona2016coin}
F.~Orabona and D.~P{\'a}l.
\newblock Coin betting and parameter-free online learning.
\newblock In \emph{Advances in Neural Information Processing Systems
  (NeurIPS)}, 2016.

\bibitem[Orabona and P{\'a}l(2021)]{orabona2021parameter}
F.~Orabona and D.~P{\'a}l.
\newblock Parameter-free stochastic optimization of variationally coherent
  functions.
\newblock \emph{arXiv:2102.00236}, 2021.

\bibitem[Paquette and Scheinberg(2020)]{paquette2020stochastic}
C.~Paquette and K.~Scheinberg.
\newblock A stochastic line search method with expected complexity analysis.
\newblock \emph{SIAM Journal on Optimization}, 30\penalty0 (1):\penalty0
  349--376, 2020.

\bibitem[Polyak(1987)]{polyak1987}
B.~T. Polyak.
\newblock \emph{Introduction to Optimization}.
\newblock Optimization Software, Inc, 1987.

\bibitem[Roughgarden(2005)]{roughgarden2005selfish}
T.~Roughgarden.
\newblock \emph{Selfish routing and the price of anarchy}.
\newblock MIT press, 2005.

\bibitem[Sadiev et~al.(2023)Sadiev, Danilova, Gorbunov, Horv{\'a}th, Gidel,
  Dvurechensky, Gasnikov, and Richt{\'a}rik]{sadiev2023high}
A.~Sadiev, M.~Danilova, E.~Gorbunov, S.~Horv{\'a}th, G.~Gidel, P.~Dvurechensky,
  A.~Gasnikov, and P.~Richt{\'a}rik.
\newblock High-probability bounds for stochastic optimization and variational
  inequalities: the case of unbounded variance.
\newblock In \emph{International Conference on Machine Learning (ICML)}, 2023.

\bibitem[Sharrock and Nemeth(2023)]{sharrock2023coin}
L.~Sharrock and C.~Nemeth.
\newblock Coin sampling: Gradient-based bayesian inference without learning
  rates.
\newblock In \emph{International Conference on Machine Learning (ICML)}, 2023.

\bibitem[Streeter and McMahan(2012)]{streeter2012no}
M.~Streeter and H.~B. McMahan.
\newblock No-regret algorithms for unconstrained online convex optimization.
\newblock In \emph{Advances in Neural Information Processing Systems
  (NeurIPS)}, 2012.

\bibitem[Vaswani et~al.(2019)Vaswani, Mishkin, Laradji, Schmidt, Gidel, and
  Lacoste-Julien]{vaswani2019painless}
S.~Vaswani, A.~Mishkin, I.~Laradji, M.~Schmidt, G.~Gidel, and
  S.~Lacoste-Julien.
\newblock Painless stochastic gradient: Interpolation, line-search, and
  convergence rates.
\newblock In \emph{Advances in Neural Information Processing Systems
  (NeurIPS)}, 2019.

\bibitem[Zhang and Cutkosky(2022)]{zhang2022parameter}
J.~Zhang and A.~Cutkosky.
\newblock Parameter-free regret in high probability with heavy tails.
\newblock In \emph{Advances in Neural Information Processing Systems
  (NeurIPS)}, 2022.

\bibitem[Zhang et~al.(2022)Zhang, Cutkosky, and Paschalidis]{zhang2022pde}
Z.~Zhang, A.~Cutkosky, and I.~Paschalidis.
\newblock {PDE}-based optimal strategy for unconstrained online learning.
\newblock In \emph{International Conference on Machine Learning (ICML)}, 2022.

\end{thebibliography}

\newpage 

\appendix

\section{Derivation of PoA upper bounds}\label{app:upper-bound-deriv}
\newcommand{\sgd}[1][\eta]{\mathsf{SGD}_{#1}}
\newcommand{\adasgd}[1][R]{\mathsf{AdaSGD}_{#1}}

In this section we explain how to go from published suboptimality bounds to the PoA bounds stated in \Cref{tab:poa-upper-lower-bounds}.
For brevity we define $g_t = \grad f(x_t; S_{t+1})$ and $\Delta_t = f(x_t; S_{t+1}) - f(x_\star; S_{t+1})$.

\paragraph*{SGD.}
Let $\sgd$ denote fixed-steps size SGD with step size $\eta$ (outputting the iterate average). Then, for $(f,P) \in \Iclass[\SMLip]^{L,R}$ it is well-known the iterates of SGD satisfy \cite[see, e.g.,][Theorem 2.13]{orabona2021modern}:
\[
\frac{1}{T} \sum_{t=0}^{T-1} \Delta_t \le \frac{\| x_0 - x_\star \|^2}{2 \eta T} + \frac{\eta}{2 T} \sum_{t=0}^{T-1} \norm{g_t}^2
\] 
where $x_\star$ is any optimal solution.
By online-to-batch conversion \cite[Chapter 3]{orabona2021modern} we get
\[
\ErrE{\sgd}{\Iclass[\SMLip]^{L,R}} \le \frac{R^2}{2\eta T} + \frac{\eta L^2}{2}.
\numberthis \label{eq:sgd-expected-err}
\]
Therefore, by the PoA definition and \Cref{prop:standard-error-bounds}
\[
\PoAE{\sgd}{\Mclass[\SMLip]^{\ell,\rho}} \le  O\prn*{\sup_{L \in [1, \ell], R \in [1, \rho]}\frac{\sqrt{T}}{L R} \prn*{ \frac{R^2}{\eta T} + \eta L^2}} \le O\prn*{\max\crl*{\frac{\rho}{\eta \sqrt{T}}, \eta \ell \sqrt{T}}}. 
\]
We can minimize the LHS of the previous expression by setting $\eta = \eta_\star \defeq \sqrt{\frac{2 \rho}{\ell T}}$ giving
\begin{equation*}
	\PoAE{\sgd[\eta_\star]}{\Mclass[\SMLip]^{\ell,\rho}} =O\prn*{\sqrt{\ell\rho}}.
\end{equation*}

\paragraph*{Adaptive SGD.} Let $\adasgd[\rho]$ denote adaptive SGD \citep{mcmahan2017survey,gupta2017unified} with projection to $\xset \cap \crl*{ x \mid \| x \| \le \rho }$. Then, we have by \cite[Section 3.3]{gupta2017unified}
\[
\sum_{t=0}^{T-1} \Delta_t \le \sqrt{2} \rho \sqrt{\sum_{t=0}^{T-1} \| g_t \|^2 }
\]
using (i) online-to-batch conversion \cite[Chapter 3]{orabona2021modern}, (ii)  Jensen's inequality on the concave function $t\mapsto \sqrt{t}$, and (iii) the bound $\E_{S\sim P} \norm{\grad f(x; S)}^2 \le L^2$ we get
\begin{flalign*}
	\ErrE{\adasgd}{\Iclass[\SMLip]^{L,R}} 
	& \overle{(i)} 
	\Ex*{\frac{\sqrt{2} \rho }{T} \sqrt{\sum_{t=0}^{T-1} \| g_t \|^2 }} 
	\overle{(ii)} 
	\frac{\sqrt{2} \rho }{T}  \sqrt{\Ex*{\sum_{t=0}^{T-1} \| g_t \|^2 }} \overle{(iii)} 
	\sqrt{2} \frac{L \rho}{\sqrt{T}},
\end{flalign*}
which the PoA definition and \Cref{prop:standard-error-bounds} implies
\begin{equation*}
	\PoAE{\adasgd[\rho]}{\Mclass[\SMLip]^{\infty,\rho}} 
	=  O\prn*{\sup_{L \in [1, \ell], R \in [1, \rho]}\frac{\sqrt{T}}{L R} \cdot \frac{L\rho}{\sqrt{T}} }
	= O(\rho). 
\end{equation*}
Thus, adaptive SGD is completely robust to unknown stochastic gradient second moment, but more sensitive to unknown domain size than SGD.

\paragraph*{\citet*{mcmahan2014unconstrained}.} Applying Theorem 11 in  \cite{mcmahan2014unconstrained} with $u = x_\star$, $a = G^2 \pi$ and $G = \ell$ yields the following regret bound for $(f, P) \in \Iclass[\Lip]^{\ell,\infty}$,
\[ 
\sum_{t=0}^{T-1} \Delta_t \le O\prn*{ \ell \| x_\star \|  \sqrt{T \log\prn*{ \frac{\ell \sqrt{T} \| x_\star \|}{\epsilon} + 1}} + \epsilon  }.
\] 
Setting $\epsilon = \ell \sqrt{T}$ and employing online-to-batch conversion for $R \ge 1$ gives
\[
\ErrE{\text{\citet{mcmahan2014unconstrained}}}{\Iclass[\Lip]^{\ell,R}} \le O\prn*{ \frac{\ell R}{\sqrt{T}} \sqrt{\log\prn*{ R + 1}} }.
\]
Hence, for all $\ell \ge 1$, $\rho \ge 2$ we have
\[
\PoAE{\text{\citet{mcmahan2014unconstrained}}}{\Mclass[\Lip]^{\ell,\rho}} = 
		O\prn*{ \ell \sqrt{ \log \rho} }.
\]
Compared to SGD, this algorithm has a much stronger adaptivity to domain size but weaker adaptivity to gradient norm, and it also requires the stronger assumption of a probability 1 gradient norm bound.

\paragraph*{\citet{cutkosky2018black}.}
For $(f,P) \in \Iclass[\Lip]^{1,\infty}$, using the result of \citet[Page 6]{cutkosky2018black} with the Euclidean norm and $\lambda = 1$ gives a regret bound of 
\[ 
\sum_{t=0}^{T-1} \Delta_t \le O\prn*{ \epsilon + \| x_\star \| \max\crl*{ \log\prn*{ \frac{\| x_\star \| G_T}{\epsilon} },  \sqrt{G_T \log\prn*{ \frac{\| x_\star \|^2 G_T}{\epsilon^2} + 1}}  } }.
\]  
with $G_T := \sum_{t=0}^{T-1} \norm{ g_t }^2$. If $(f,P) \in \Iclass[\Lip]^{\ell,\infty}$ then using this method gradients rescaled by $1/\ell$ yields 
\[ 
\sum_{t=0}^{T-1} \Delta_t \le   O\prn*{ \ell \epsilon + \| x_\star \| \max\crl*{ \ell \log\prn*{ \frac{\| x_\star \| G_T}{\epsilon \ell^2} },  \sqrt{G_T \log\prn*{ \frac{\| x_\star \|^2 G_T}{\epsilon^2 \ell^2} + 1}}  } }.
\]  
Setting $\epsilon = \sqrt{T}$ and employing online-to-batch conversion for $L \in [1, \ell]$ and $R \in [1, \rho]$ gives
\[
\ErrE{\text{\citet{mcmahan2014unconstrained}}}{\Iclass[\Lip]^{L,R}} \le O\prn*{ \frac{\ell}{T} \prn*{ 1 + R \log( \rho ) } + \frac{L R}{\sqrt{T}} \sqrt{\log( \rho + 1 )} }.
\]
Therefore,
\[
\PoAE{\text{\citet{cutkosky2018black}}}{\Mclass[\Lip]^{\ell,\rho}} = 
	O\prn*{ \frac{\ell}{\sqrt{T}}\log \prn*{\rho} + \sqrt{\log \prn*{ \rho + 1} } }.
\]
\paragraph*{\text{\citet{carmon2022making}}.} Applying Theorem 14 in \cite{carmon2022making}
with $\eta_{\epsilon} = \frac{1}{\ell B}$, $L = \ell$, $\delta \in (0,1/2)$ we get for all $\hat{L} \in [1, \ell]$ and $R \ge 1$ that for all $B \in \N$
\begin{flalign*}
\ErrHP[B]{\text{\citet{carmon2022making}}}{\Iclass[\Lip]^{\hat{L},R}} &\le O\prn*{ R \frac{\sqrt{C \hat{L}^2 ( B / M ) + C^2 \ell }}{B / M } + \frac{\frac{1}{\ell B} \prn*{C \hat{L}^2 ( B / M ) + C^2 \ell^2}}{B / M} } \\
&\le O\prn*{ \frac{(C + M) R}{\sqrt{B}} \prn*{ \hat{L} + \frac{(C+M) \ell }{\sqrt{B}} }} 
\end{flalign*}
where $C = \log \frac{\log\prn*{1 + R }}{\delta}$ and $M = \min\{ B , \log\log\prn*{T (1 + R) } \}$.
Therefore, substituting for $C$ and $M$,  and using $\rho \ge \max\{ R, 2\}$ gives that the previous display is upper bounded by
\begin{flalign*}
	\hspace{48pt} & \hspace{-48pt} O\prn*{ \prn*{\log \frac{\log\prn*{\rho T}}{\delta}} \prn*{ \frac{\hat{L} R}{\sqrt{B}} + \prn*{\log \frac{\log\prn*{\rho T}}{\delta}} \frac{\ell R}{B} } } 
	= O\prn*{ \prn*{\log \frac{\log\prn*{\rho T}}{\delta}}^2\prn*{ \frac{\hat{L} R}{\sqrt{B}} + \frac{\ell R}{B} }}.
\end{flalign*}
\ollie{actually if we were careful with $\eta$ i think we could just make it a single loglog power or a loglog times a logloglog}
Therefore,
\[
\PoAHP{\text{\citet{carmon2022making}}}{\Mclass[\Lip]^{\ell,\rho}} = O\prn*{ \prn*{\log \frac{\log\prn*{\rho T}}{\delta}}^2 \left( 1+ \frac{\ell}{\sqrt{T}} \right) }.
\]
In a similar manner, \cite[Section~D.2]{carmon2022making} gives an $\O{ \log\prn*{\frac{\log(\rho T)}{\delta}} }$ bound on $\PoAHP{\cdot}{\Mclass[\Lip]^{1,\rho}}$. 

\paragraph*{\citet{cutkosky2019artificial}.} %
\citet[Page 7]{cutkosky2019artificial} provides an algorithm which, for $D\ge 1$ and domain $\xset \cap \{ x : \| x \| \le D \}$, enjoys a regret guarantee of
\[ 
\sum_{t=0}^{T-1} \Delta_t \le \Otilb{ \| x_\star \| G \sqrt{T}  + D G },
\]
where $G = \max\crl*{ 1, \max_{t<T} \| \grad f(x_t, S_{t+1}) \| }$.
Setting $D = \rho$ and applying online-to-batch conversion this gives for all $R \in [1, \rho]$ that
\[
\ErrE{\text{\citet[Page 7]{cutkosky2019artificial}}}{\Iclass[\Lip]^{L,R}} \le \Otilb{ \frac{LR}{\sqrt{T}} +  \frac{L \rho}{T} }.
\]
Thus, 
\[
\PoAE{\text{\citet[Page 7]{cutkosky2019artificial}}}{\Mclass[\Lip]^{\ell,\rho}} = \Otilb{ 1 +  \frac{\rho}{\sqrt{T}} }.
\]
Notably, unlike the other algorithms 
in \Cref{tab:poa-upper-lower-bounds},
to obtain this result we needed to know $\rho$. 
When $\rho$ is unknown 
currently the best bound on the price of adaptivity is $\tilde{O}( 1 + \rho^2 / \sqrt{T})$ \cite[Corollary 4]{cutkosky2019artificial}.

\paragraph*{\citet{zhang2022parameter}.} 
For any $(f,P)\in\Iclass[\MSLip]^{\ell,\infty}$ we have $\E_{S \sim P} \norm{ \grad f(x_t, S) - \grad F_{f,P}(x_t)}^2 \le \E_{S \sim P} \norm{ \grad f(x_t, S) }^2  \le \ell^2$ and $\norm{ \grad F_{f,P}(x_t) } \le \ell$. Therefore, Theorem 4 of \citet{zhang2022parameter} with $p=2$ gives that,
\[
\E\brk*{ \sum_{t=0}^{T-1} \Delta_t } = \Otilb{ \epsilon + \ell \| x_0 - \xopt \| T^{1/2} },
\]
where we are hiding logarithmic dependence in $\epsilon$ and other problem parameters. Taking $\epsilon = \sqrt{T}$, we conclude that for every $R\ge 1$ and $L\in [1,\ell]$ we have
\[
\ErrHP{\text{\citet{zhang2022parameter}}}{\Iclass[\MSLip]^{L,R}} = \Otilb{ \frac{\ell R}{\sqrt{T}} }.
\]
Therefore, we have $\PoAHP{\text{\citet{zhang2022parameter}}}{\Mclass[\MSLip]^{\ell,\rho}} =  \Otilb{ \ell }$.
\section{Proof of \Cref{prop:PoA-props}}\label{app:PoA-props-proof}

\restatePropPoAProps*

\begin{proof}
	\Cref{item:PoA-props:trivial-lb} is a direct consequence of definition~\eqref{eq:PoAGen-def} of the PoA and the definition~\eqref{eq:minimax-err-def} of minimax error, which implies that $\Err{\alg}{\Iclass} \ge \Err{\Aclass}{\Iclass}$ for all $\alg\in\Aclass$.
	
	\Cref{item:PoA-props:trivial-ub} holds since for any $(f,P)\in \Iclass$ and $\Iclass\in\{\Iclass[\Lip]^{L,R},\Iclass[\SMLip]^{L,R}\}$ the objective function $F_{f,P}$ is $L$ Lipschitz. If $\xopt$ is the minimizer of $F_{f,P}$ closest to $x_0=0$ then we have $F_{f,P}(x_0)-F_{f,P}(\xopt)\le L\norm{\xopt-x_0} \le LR$. Therefore, $\Err{\alg_0}{\Iclass}\le LR$ and the resulting upper bounds follow from \Cref{prop:standard-error-bounds}. 
	
	To show \Cref{item:PoA-props:expectation-higher-PoA}, note that for any nonnegative random variable $Z$ and any $\delta \in (0,1)$ we have $\E Z \ge \delta Q_{1-\delta} (Z)$ by Markov's inequality and consequently, for any $\alg$ and instance class $\Iclass$ we have $\ErrE{\alg}{\Iclass} \ge \delta \ErrHP{\alg}{\Iclass}$. Moreover, by \Cref{prop:standard-error-bounds} we have that $\ErrE{\Aall}{\Iclass} \le O\prn*{\ErrHP{\Aall}{\Iclass}}$ for any $\Iclass\in\{\Iclass[\Lip]^{L,R},\Iclass[\SMLip]^{L,R}\}$ and $\delta\in(0,\half)$. Therefore, for $\Mclass\in\{\Mclass[\Lip]^{\ell,\rho}, \Mclass[\SMLip]^{\ell,\rho}\}$ we have
	\begin{equation*}
		\PoAE{\alg}{\Mclass} = \sup_{\Iclass \in \Mclass}
		\frac{\ErrE{\alg}{\Iclass}}{\ErrE{\Aall}{\Iclass}}
		=
		\Omega\prn*{ \delta \sup_{\Iclass \in \Mclass} \frac{\ErrHP{\alg}{\Iclass}}{\ErrHP{\Aall}{\Iclass} } }
		=
		\Omega(\delta) \PoAHP{\alg}{\Mclass}
		.
	\end{equation*}
	
	\Cref{item:PoA-props:MS-higher-PoA} follows from the fact that $\Iclass[\Lip]^{L,R} \subset \Iclass[\SMLip]^{L,R}$ and therefore 
	$\Err{\alg}{\Iclass[\Lip]^{L,R}} \le \Err{\alg}{\Iclass[\SMLip]^{L,R}}$ for every $\alg$, combined with \Cref{prop:standard-error-bounds} which implies
	$\Err{\Aall}{\Iclass[\Lip]^{L,R}} = \Theta(1) \Err{\Aall}{\Iclass[\SMLip]^{L,R}}$.
	
	\Cref{item:PoA-props:All-Afo-same-PoA} is immediate from the fact $\AFO$ has the same minimax complexity as $\Aall$ for all the instance classes in $\Mclass\in\{\Mclass[\Lip]^{\ell,\rho}, \Mclass[\SMLip]^{\ell,\rho}\}$, as shown in \Cref{prop:standard-error-bounds}.
	
	Finally, to show \Cref{item:PoA-props:scaling}, consider $\alg'$ that computes $\bar{x}_T$ by applying $\alg$ on the modified sample function $\bar{f}(x;s) = \Lmin \Rmin f(x/\Rmin; s)$, and then returns $x_T=\bar{x}_T/\Rmin$. Clearly, for any distribution $P$, 
	\[
		F_{\bar{f},P}(\bar{x}_T) - \inf_{x'\in \Rmin \xset} F_{\bar{f},P}(x')= \Lmin \Rmin \prn*{
			F_{f,P}(x_T)-\inf_{x'\in \xset}F_{f,P}(x')
		}
	\]
	and therefore 
	\[
	 \Err{\alg}{(\bar{f},P)} = \Lmin \Rmin \Err{\alg'}{(f,P)}.
	\]
	Moreover, for any $\alpha,\beta>0$, the instance $(f,P)\in \Iclass[\Lip]^{\alpha, \beta}$ if and only if $(\bar{f},P)\in \Iclass[\Lip]^{\alpha \Lmin, \beta \Rmin}$ and therefore,
	\[
	 \Err{\alg}{\Iclass[\Lip]^{\alpha \Lmin, \beta \Rmin}} = \Lmin \Rmin \Err{\alg'}{\Iclass[\Lip]^{\alpha, \beta}}.
	\]
	Further minimizing over $\alg$ shows that
	\[
	 \Err{\Aall}{\Iclass[\Lip]^{\alpha \Lmin, \beta \Rmin}} = \Lmin \Rmin \Err{\Aall}{\Iclass[\Lip]^{\alpha, \beta}}.
	\]
	Substituting into the definition of the PoA, we have
	\begin{flalign*}
	\PoA{\alg}{\Mclass[\Lip]^{\Lmin,\Lmax,\Rmin,\Rmax}} 
	&= \sup_{\alpha\in\brk*{\Lmin,\Lmax},\beta\in\brk*{\Rmin,\Rmax}} \frac{
		\Err{\alg}{\Iclass[\Lip]^{\alpha, \beta}}}{\Err{\Aall}{\Iclass[\Lip]^{\alpha, \beta}}} \\
	&= \sup_{\alpha\in\brk*{1,\frac{\Lmax}{\Lmin}},\beta\in\brk*{1,\frac{\Rmax}{\Rmin}}} \frac{\Err{\alg}{\Iclass[\Lip]^{\alpha \Lmin, \beta \Rmin}}}{\Err{\Aall}{\Iclass[\Lip]^{\alpha\Lmin, \beta\Rmin}}} \\
	&= \sup_{\alpha\in\brk*{1,\frac{\Lmax}{\Lmin}},\beta\in\brk*{1,\frac{\Rmax}{\Rmin}}} \frac{\Lmin\Rmin\Err{\alg'}{\Iclass[\Lip]^{\alpha, \beta}}}{\Lmin\Rmin\Err{\Aall}{\Iclass[\Lip]^{\alpha, \beta}}} %
	= \PoA{\alg'}{\Mclass[\Lip]^{\nicefrac{\Lmax}{\Lmin}, \nicefrac{\Rmax}{\Rmin}}}.
	\end{flalign*}
	The same argument applies when we replace $\textsf{Lip}$ with $\textsf{\SMLip}$ everywhere.
	\end{proof} 
\end{document}